\documentclass[final]{siamltex}
\usepackage{animate}
\usepackage{amsmath}
\usepackage{amssymb}
\usepackage{graphics}
\usepackage{graphicx}
\usepackage{footnote}
\usepackage{textcomp}
\usepackage{mathrsfs}
\usepackage{epstopdf}
\usepackage{array}
\usepackage{showkeys}
\usepackage[maxfloats=99]{morefloats}
\usepackage{url}
\usepackage{cases}
\usepackage{mathscinet}
\usepackage[normalem]{ulem}
\usepackage{algorithmic, algorithm}
\usepackage{color,xcolor}
\usepackage{cite}
\usepackage{subcaption}

\usepackage{stmaryrd}
\usepackage{bm}

\graphicspath{{fig/}}

\colorlet{DarkGreen}{green!70!black}

\numberwithin{equation}{section}

\newtheorem{remark}[theorem]{Remark}

\title{Superconvergent Gradient Recovery   for Virtual Element Methods}
\author{Hailong Guo\thanks{School of Mathematics and Statistics,  The University of Melbourne,  Parkville, VIC 3010, Australia   (hailong.guo@unimelb.edu.au).
This work was partially supported by Andrew Sisson Fund of the University of Melbourne. }
\and%
 Cong Xie\thanks{College of Mathematics and Systems Science, Xinjiang University, Urumqi, 830046, P.R. China and
College of Mathematics and Physics, Hebei University of Architecture, Zhangjiakou, 075000, P. R. China(xiecong121@163.com).}
 \and%
 Ren Zhao\thanks{School of Mechanical Engineering and Automation, Harbin Institute of Technology,
  Shenzhen, 518055, P. R. China (zhaoren@hit.edu.cn
).}}

\begin{document}

\maketitle

%
%
\medskip

\begin{abstract}
Virtual element method is a new promising finite element  method using general polygonal meshes.  Its optimal {\it a priori} error estimates are
well established  in literature.   In this paper, we take a different viewpoint.   We try to  uncover the superconvergent
property of  virtual element methods by doing some local  post-processing only  on the degrees of freedom.   Using the linear virtual element method
as an example, we propose a universal gradient recovery procedure to improve the accuracy of gradient approximation for numerical methods using general
polygonal meshes.  Its capability of serving  as {\it a posteriori} error estimators in adaptive computation is also investigated.  Compared to the existing
residual-type {\it a posteriori} error estimators for the  virtual element methods, the recovery-type {\it a posteriori} error estimator based on the proposed
gradient recovery technique is much  simpler in implementation and it is asymptotically exact.   A series of  benchmark tests are presented to numerically
illustrate the superconvergence of recovered gradient and validate the asymptotic exactness of the recovery-based {\it a posteriori }error estimator.
\vskip .3cm

\end{abstract}


\section{Introduction}
\label{sec:int}
  The idea of using polygonal elements can be  traced back to Wachspress \cite{Wa1975}.    After that, there has been tremendous interest in developing finite element/difference  methods using general polygons, see the review paper \cite{MRS2014} and the references therein.   Well-known examples include
  the  polygonal finite element methods\cite{ST2004, SM2006},   mimetic finite difference methods\cite{BLM2014,Sh1996,KR2003, HS1999, SS1996},  hybrid high-order methods\cite{DEL2014, DE2015},  polygonal discontinuous Galerkin methods \cite{MWWY2014}, etc.

Virtual element methods  evolve from the mimetic finite difference methods \cite{BBL2009,BLM2011} within the framework of the finite element methods.
It was first proposed for the Poisson equations \cite{BBCMMR2013}. Thereafter,  it has been developed  to many other equations \cite{BBMR2016, CMS2017, ADLP2017, BM2013}. It generalizes the classical finite element methods    from simplexes   to general polygons/polyhedrons including non-convex ones.  This enables the virtual element methods with the capability of dealing with polygons/polyhedrons with arbitrary numbers
of edges/faces and coping with more general  continuity.  This makes the virtual element methods handle hanging nodes naturally and simplifies the procedure of adaptive mesh refinement.   Different from other polygonal finite element methods, the non-polynomial basis functions are never explicitly constructed and   evaluating non-polynomial functions
is totally unnecessary. Consequently, the only available data in virtual element methods  are the degrees of freedom.
  The optimal convergence theory  was well established in \cite{BBCMMR2013,BBMR2016}.

In many cases,  the gradient of a solution attracts much more attention than the solution itself.  That is due to two different aspects: (i) gradient
has physical meaning like momentum, pressure, et. al; (ii) many problems, like the free boundary value problems and moving interface problems, depend on the first order derivatives of the solution.  For virtual element methods, like their predecessors: standard finite element methods, the gradient approximation accuracy is one order lower than the corresponding solution approximation accuracy. Thus, a more accurate approximate gradient is highly desirable in scientific and engineering computing.

For finite element methods on triangles or quadrilaterals, it is well-known  that gradient recovery  is one of the most important post-processing procedures  to reconstruct a more accurate approximate gradient than the finite element gradient.      The gradient recovery methods are well developed for the classical finite element methods and there are
a massive number of references, to name a few \cite{LMW2000, ZZ1987, ZZ1992a, ZZ1992b, BX2003, ZN2005, NZ2005, GZ2015, XZ2004, GZZ2016b}.
Famous examples include the simple/weighted averaging \cite{ZZ1987}, superconvergent patch recovery \cite{ZZ1992a, ZZ1992b} (SPR), and the polynomial preserving
recovery \cite{ZN2005, NZ2005, NZ2004} (PPR).
Right now, SPR and PPR become standard tools in modern scientific and engineering computing. It is evident by the fact that
SPR is available in many commercial finite element software like ANSYS, Abaqus, LS-DYNA, and  PPR is included in COMSOL Multiphysics.

The first purpose of this paper is to introduce a gradient recovery technique as a post-processing procedure for the linear virtual element method and uncover its superconvergence property.
To recover the gradient on a general polygonal mesh,  the most straightforward idea is to take simple averaging or weighted averaging.  However, we will encounter two difficulties:
first, the values of the gradient  are not computable in the linear virtual element method;  second,  the consistency of the simple averaging or weighted averaging methods depends
strongly on the symmetry of the local patches and sometimes they are inconsistent even on some uniform meshes.   To overcome the first difficulty, one may simply replace  the virtual
element gradient by its polynomial projection.   Then, we are able to apply the simple averaging or weighted averaging methods to   the projected virtual element
gradients.  But we may be at risk of introducing some additional error and computational cost.     Similarly, if we want to generalize SPR to general polygonal meshes,  we also have those two difficulties.  The second difficulty is more severe since  there is no longer  local  symmetric property   for polygonal meshes.
To tackle those difficulties,  we generalize the idea of PPR \cite{ZN2005} to the general polygons, which only uses the degrees of freedom and
 has the consistency on arbitrary polygonal  meshes  by the polynomial preserving property.    We prove the polynomial preserving and boundedness properties of
the generalized gradient recovery operator.  Moreover, the superconvergence of the recovered gradient using the interpolation of the exact solution is theoretically justified.
We also numerically  uncover the superconvergent property of the linear virtual element method.   The recovered gradient is numerically proven to be more accurate
than the virtual element gradient.
The post-processing procedure  also provides a way to visualize the gradient filed which is not directly available in the virtual element methods.

Adaptive computation is an essential tool in scientific and engineer computing. Since the pioneering work of Babu\v{s}ka and Rheinboldt \cite{BR1978} in the 1970s,
there has been a lot of effort devoted to both the theoretical development
of adaptive algorithms and applications of adaptive finite element methods. For classical finite element methods, adaptive finite element methods
have reached a stage of maturity, see the monographs \cite{AO2000, Re2008, Ve2013, BS2001} and the references therein.  For adaptive finite element methods, one of the key ingredients is   to design  {\it a posteriori} error estimators.  In the literature, there are two types of {\it a posteriori} error estimators: residual-type and recovery-type.

For virtual element methods, there are only a few work concerning on the {\it a posteriori }error estimation and adaptive algorithms.
In \cite{BM2015},  Beir\~{a}o Da Veiga and Manzini derived  {\it a posteriori }error estimators for $C^1$ virtual element methods.
In \cite{CGPS2017}, Cangiani et al. proposed {\it a posteriori} error estimators for the $C^0$ conforming virtual element methods
for solving second order general elliptic equations.
In \cite{BB2017},  Berrone and Borio derived a new {\it a posteriori} error estimator for  the $C^0$ conforming virtual element methods using
the projection of the virtual element solution.
In \cite{MRR2017},   Mora et el. conducted {\it a posteriori} error analysis for a virtual element method for the Steklov eigenvalue problems.  All the above {\it a posteriori} error estimators are residual-type.   To the best of our knowledge,  there is no recovery-type
{\it a posteriori} error estimators for virtual element methods yet.

The second purpose of this paper is to present a recovery-type {\it a posteriori} error estimator for the linear virtual element method.
But for the virtual element methods, there is no explicit formulation for the basis functions.   To construct a fully computable
{\it a posteriori} error estimator,   we propose to use the gradient of polynomial  projection of  virtual element solution subtracting
the polynomial projection of the recovered gradient.   Compared with the existing residual-type {\it a posteriori} error estimators  \cite{BM2015, CGPS2017, BB2017},
the error estimator has only one term and hence it is much simpler.    The error estimator is numerically proven to be asymptotically exact,
which makes it more favorable  than other  {\it a posteriori} error estimators for virtual element methods.

The rest of the paper is organized as follows. In Section 2,  we introduce  the model problem and  related notations.
In Section 3,  we present the construction of  the linear virtual element space and the definition of  discrete formulation of the problem.
In Section 4, we propose the gradient recovery procedure and prove the consistency and boundedness of the proposed
gradient recovery operator.  The recovery-based {\it a posteriori} error estimator is constructed in Section 5.
In Section 6,  the superconvergent property of the gradient recovery operator and the asymptotic exactness of the
recovery-based error estimator is numerically verified.  Some conclusions are drawn in Section 7.

\section{Model Problems}
\label{sec:model}
Let $\Omega\subset \mathbb{R}^2$ be a bounded polygonal domain with Lipschitz boundary $\partial\Omega$.
 Throughout  this paper, we adopt the standard
notations for Sobolev spaces and their associate norms given in \cite{BS2008, Ci2002, Evans2008}.
For a subdomain $\mathcal{D}$ of $\Omega$,  let $W^{k,p}(\mathcal{D})$ denote the Sobolev space with norm
$\|\cdot\|_{k, p, \mathcal{D}} $ and seminorm $|\cdot|_{k, p,\mathcal{D}}$.
 When $p = 2$, $W^{k,2}(\mathcal{D})$ is simply  denoted by $H^{k}(\mathcal{D})$
 and the subscript $p$ is omitted in its associate norm and seminorm.
 $(\cdot, \cdot)_{\mathcal{D}}$ denotes the standard $L_2$ inner product on $\mathcal{D}$ and the subscript
 is ignored when $\mathcal{D} = \Omega$.
Let $\mathbb{P}_m(\mathcal{D})$ be the space of polynomials of
degree less than or equal to $m$ on $\mathcal{D}$ and $n_m$ be the
dimension of $\mathbb{P}_m(\mathcal{D})$ which equals to $\frac{1}{2}(m+1)(m+2)$.

Our  model problem is the following Poisson equation
\begin{align}
 -\Delta u=f &\text{  in }\Omega,\label{equ:model}\\
    u=0 & \text{  on }\partial\Omega.\label{equ:bndcond}
\end{align}
The homogeneous Dirichlet boundary condition is   considered for the sake of clarity.  Inhomogeneous Dirichlet
and other types of boundary conditions  apply as well without  substantial modification.

Define the bilinear form $a(\cdot, \cdot):  H^1(\Omega)\times H^1(\Omega)\rightarrow \mathbb{R}$ as
\begin{equation}\label{equ:bilinear}
a(u,v) = (\nabla u, \nabla v),
\end{equation}
for any $u, v \in H^1(\Omega)$.    It is easy to see
that $|v|_{1, \Omega}^2 = a(v,v)$ and $|\cdot|_{1, \Omega}$ is a norm on $H_0^1(\Omega)$ by
the Poincar\'e inequality.

The variational formulation of \eqref{equ:model}  and \eqref{equ:bndcond} is to find  $u\in H^1_0(\Omega)$ such that
\begin{equation}\label{equ:vf}
a(u,  v) = (f, v), \quad \forall v\in H^1_0(\Omega).
\end{equation}
Lax-Milgram theorem implies it admits a unique solution.

\section{Virtual Element Method}
\label{sec:vem}

Let $\mathcal{T}_h$ be a partition of $\Omega$ into non-overlapping polygonal elements $E$ with non-self-intersecting polygonal boundaries.
Let $h_E$ be the diameter of element $E$ and $h = \max_{E\in\mathcal{T}_h} h_E$.  Throughout this paper, we assume that there exists $\rho \in (0,1)$ such that
 the mesh $\mathcal{T}_h$
satisfies the following two assumptions \cite{BGS2017, BBMR2016}:
\begin{itemize}
\item [(i).]  every element $E$ is star-shaped with respect to every point of a disk $D$ of radius $\rho h_E$;
\item [(ii).] every edge $e$ of $E$ has length $|e| \ge \rho h_E$.
\end{itemize}

In this paper,  we focus on  the lowest  order virtual element method as in \cite{AHMAD2013376, BBMR2014}.   To define the virtual element space, we begin with defining the
local virtual element spaces on each element. For such purpose,  let
\begin{equation}
\mathbb B(\partial E):=\{v\in C^0(\partial E):v|_e\in\mathbb P_1(e),\quad \forall e\in \partial E\}.
\end{equation}
Then, the local virtual element space $V(E)$ on the element $E$ can be defined as
\begin{equation}
V(E)=\{v\in H^1(E):v|_{\partial E}\in\mathbb B(\partial E),\quad\Delta v|_E \in \mathbb{P}_1(E) \}.
\end{equation}

The soul of virtual element methods is that the non-polynomial basis functions are never explicitly constructed  and needed.  This is made possible
by introducing the  projection operator $\Pi^{\nabla}$.  For any function $v_h \in V(E)$,  its projection  $\Pi^{\nabla}v_h$ is defined to satisfy the following
orthogonality:
\begin{equation}\label{equ:orthproj}
(\nabla p, \nabla (\Pi^{\nabla} v_h - v_h))_{ E} = 0, \quad \forall p\in \mathbb{P}_1(E),
\end{equation}
plus(to take care of the constant part of $\Pi^{\nabla}$):
\begin{equation}
\int_{\partial E}(\Pi^{\nabla} v_h-v_h)ds=0.
\end{equation}

The modified local virtual element space  \cite{AHMAD2013376} is defined as
\begin{equation}
W(E)=\{v_h\in V(E): (v_h- \Pi^{\nabla}v_h, q) = 0, \forall q\in \mathbb{P}_1(E)\}.
\end{equation}
Then,  the  virtual element space  \cite{AHMAD2013376, BBMR2014} is
\begin{equation}
V_h=\{v\in H^1(\Omega):v|_{ E}\in W(E), \quad\forall E\in \mathcal{T}_h\}.
\end{equation}
The degrees of freedom in $V_h$ are only the values of $v_h$ at all vertices.  Furthermore, let $V_{h,0} = V_h\cap H^1_0(\Omega)$
be the subspace of $V_h$ with homogeneous boundary condition.

Similarly, we can define the $L_2$ projection $\Pi^0$  as
\begin{equation}\label{equ:l2proj}
( p, \Pi^{0} v_h - v_h)_{ E} = 0, \quad \forall p\in \mathbb{P}_1(E).
\end{equation}
For the linear virtual element method\cite{AHMAD2013376, BBMR2014}, these two projections are equivalent, i.e. $\Pi^{\nabla} =\Pi^0$.
In the subsequent, we will make no  distinction between these two projections.

On each element $E\in \mathcal{T}_h$, we can define the following discrete bilinear form
\begin{align}
a^E_h(u_h,v_h)=(\nabla \Pi^{\nabla} u_h, \nabla \Pi^{\nabla} v_h)_E+S^E(u_h-\Pi^{\nabla} u_h, v_h-\Pi^{\nabla} v_h)
\end{align}
for any $ u_h,v_h\in V(E)$.  The discrete bilinear form $S^E$ is symmetric and positive, which is also fully computable using only
the degrees of freedom of $u_h$.    The readers are referred to \cite{AHMAD2013376,BBCMMR2013, BBMR2014} for the detail definition of $S^E$.

Then, we can define the  discrete bilinear form $a_h(\cdot, \cdot)$:
\begin{align}
a_h(u_h,v_h)=\sum_{E\in\mathcal{T}_h}a^E_h(u_h,v_h),
\end{align}
for any $ u_h,v_h\in V_h$.
The linear virtual element method for the model  problem \eqref{equ:model} is to  find $u_h \in V_{h, 0}$ such that
\begin{align}
a_h(u_h,v_h)=(f,\Pi^0 v_h),\quad \forall v_h \in V_{h,0}.
\end{align}


\section{Superconvergent Gradient Recovery}
\label{sec:sgr}
In this section, we present a high-accuracy and efficient post-processing technique for  the  virtual element
methods. Our idea is to generalize the polynomial preserving recovery \cite{ZN2005} to general polygonal meshes.
The generalized method works for a large class of numerical methods based on polygonal meshes including
mimetic finite difference methods\cite{Sh1996, BLM2014}, polygonal finite element methods\cite{ST2004},  and virtual element methods\cite{BBCMMR2013,BBMR2016}.
To illustrate the main idea, we take the  virtual element methods as an example to demonstrate the proposed algorithm.

We focus on the linear virtual element method.   Let $V_h$  be the linear virtual element space on a  general polygonal mesh $\mathcal{T}_h$ as defined in the previous section.
The sets of all vertices  and  of all edges of the polygonal mesh $\mathcal{T}_h$ are  denoted by $\mathcal{N}_h$ and $\mathcal{E}_h$, respectively.
Let $I_h$ be the index set of $\mathcal{N}_h$.

The proposed gradient recovery  is formed  in three steps: (1) construct  local patches of elements;
(2) conduct local recovery procedures; (3) formulate the recovered data in a global expression.

To  construct a local patch,  we first construct a union of mesh elements around a vertex.
For each vertex $z_i\in \mathcal{N}_h$ and nonnegative integer $n\in \mathbb{N}$,  define $\mathcal{L}(z_i,n)$ as
\begin{equation}
\mathcal{L}(z_i,n)=
\begin{cases}
z_i, & \text{if}\ n=0,\\
\bigcup\{E: E\in\mathcal{T}_h,\ E\cap\mathcal{L}(z_i,0)\neq\phi\}, &
\text{if}\ n=1,\\
\bigcup\{E: E\in\mathcal{T}_h,\ E\cap\mathcal{L}(z_i,n-1)\
\text{is a edge in } \mathcal{E}_h\}, & \text{if}\ n\ge 2.
\end{cases}
 \label{local}
\end{equation}

\begin{figure}[htp]
   \centering
   \subcaptionbox{\label{fig:0thlayer}}
  {\includegraphics[width=0.32\textwidth]{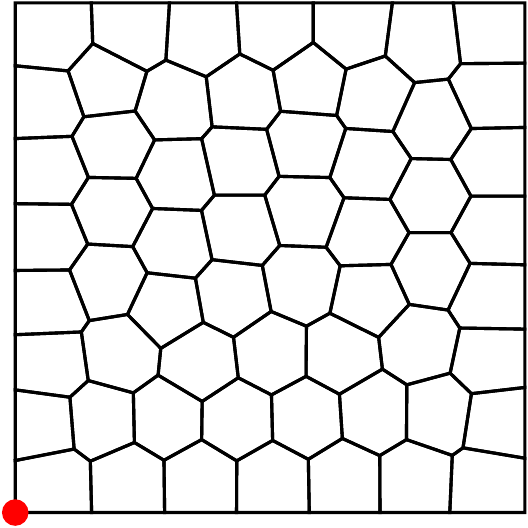}}
  \subcaptionbox{\label{fig:1thlayer}}
   {\includegraphics[width=0.32\textwidth]{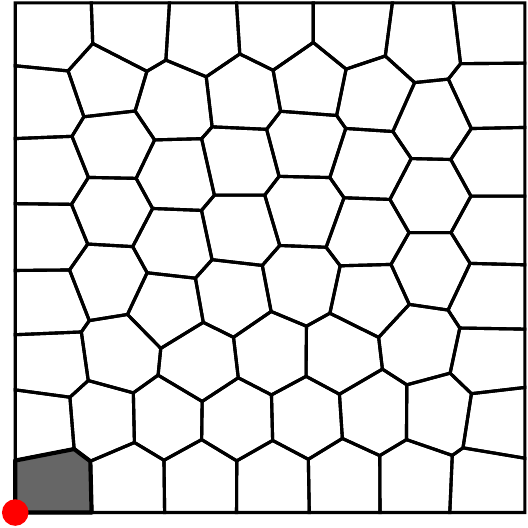}}
    \subcaptionbox{\label{fig:2thlayer}}
   {\includegraphics[width=0.32\textwidth]{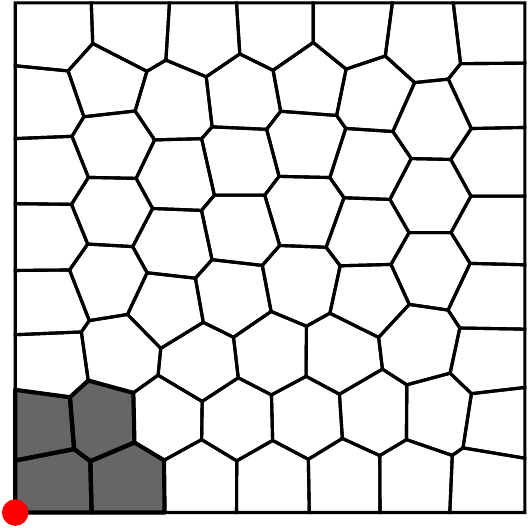}}
   \caption{Illustration of definition of $\mathcal{L}(z_i, n)$: (a) Plot of $\mathcal{L}(z_i, 0)$; (b) Plot of $\mathcal{L}(z_i, 1)$; (c) Plot of  $\mathcal{L}(z_i, 2)$.}
   \label{fig:localpatch}
\end{figure}

It is easy to see that  $\mathcal{L}(z_i, n)$ consists of the mesh elements in the first n layers around the vertex $z_i$.
In  Figure \ref{fig:localpatch}, we give an illustration of $\mathcal{L}(z_i, n)$ where $z_i$ is the red dotted point.
From the figure, we can clearly observe that $\mathcal{L}(z_i, 0)$ just contains the vertex $z_i$ itself and  $\mathcal{L}(z_i, 1)$ consists of
the elements which have $z_i$ as a vertex; while  $\mathcal{L}(z_i, 2)$ is the union of all elements in $\mathcal{L}(z_i,1)$ and their neighbourhood elements.

Let  $\Omega_{z_i}= \mathcal{L}(z_i, n_i)$  with $n_i$ be the smallest integer  such that $\mathcal{L}(z_i, n_i)$ satisfies the rank condition in the following sense:

\begin{definition}\label{def:rank}
 A local patch  $\Omega_{z_i}$ is said to satisfy the rank condition  if it admits
 a unique least-squares fitted polynomial $p_{z_i}$ in \eqref{equ:pprls}.
\end{definition}

\begin{remark}\label{rmk:pathc}
 For virtual element methods,  we are more interested in the case that $\mathcal{T}_h$ consists of polygons with more than four vertices.
 In general,  to guarantee the rank condition  we need  $n_i=1$ for interior vertices and $n_i=2$ for boundary vertices.
\end{remark}

\begin{remark}\label{rmk:bndpatch}
 For  boundary vertices, there are alternative  ways to construct the local patch satisfying the rank condition in Definition \ref{def:rank}.  The readers
 are referred to \cite{GZZZ2016}.
\end{remark}

To construct the recovered gradient at a given vertex $z_i$, let $B_{z_i}$ be the set of vertices
in $\Omega_{z_i}$ and $I_i$ be the indexes of the $B_{z_i}$.  Using the vertices in $B_{z_i}$ as
sampling points, we fit a quadratic polynomial  $p_{z_i}$ at the vertex $z_i$ in the following
least-squares sense: 
\begin{equation}\label{equ:pprls}
p_{z_i}(z)=\arg\min_{p\in \mathbb{P}_2(\Omega_{z_i})} \sum_{j\in I_i}| {p(z_{i_j})-u_{h,j}}|^2,
\end{equation}
 where $u_{h,j} = u_h(z_{i_j})$.

To avoid numerical instability in real numerical computation,  let $$h_i= \max\{|z_{i_k}-z_{i_j}|:
i_k, i_j\in I_i\},$$
and define the local coordinate transform
\begin{equation}
F: (x,y) \rightarrow (\xi, \eta) = \frac{(x,y)-(x_i,y_i)}{h_i},
\end{equation}
where $z=(x,y)$ and $\hat{z}=(\xi, \eta)$.
All the computations are performed at the reference local element patch
$\hat{\Omega}_{z_i} = F(\Omega_{z_i})$.
Then we can rewrite $p_{z_i}(z)$ as
\begin{equation}\label{equ:polydef}
p_{z_i}(z) =\bm{p}^T\bm{a} = \hat{\bm{p}}^T\hat{\bm{a}},
\end{equation}
where
\begin{align*}
&\bm{p}^T = (1, x, y, x^2, xy, y^2), \quad \hat{\bm{p}}^T = (1, \xi, \eta, \xi^2, \xi\eta, \eta^2), \\
&\bm{a} = (a_1, a_2, a_3, a_4, a_5, a_6), \quad \hat{\bm{a} }= (a_1, h_ia_2, h_ia_3, h_i^2a_4, h_i^2a_5, h_i^2a_6).
\end{align*}
Let $\hat{z}_{i_j} = F(z_{i_j})$. The coefficient $\hat{\bm{a}}$ is determined by solving  the linear system
\begin{equation}\label{equ:lsapp}
 (\hat{A}^T \hat{A}) \hat{\bm{a}} =   \hat{A}^T \bm{b},
\end{equation}
where
\begin{equation*}
\hat{A}=
\left(
\begin{matrix}
1 & \xi_{i_1} &\eta_{i_1} & \xi_{i_1} ^2&\xi_{i_1} \eta_{i_1} &\eta_{i_1} ^2 \\
1 & \xi_{i_2} &\eta_{i_2} & \xi_{i_2} ^2&\xi_{i_2} \eta_{i_2} &\eta_{i_2} ^2 \\
\vdots &\vdots& \vdots& \vdots& \vdots& \vdots\\
1 & \xi_{|I_i|} &\eta_{|I_i|} & \xi_{|I_i|}^2& \xi_{|I_i|}\eta_{|I_i|} &\eta_{|I_i|}^2 \\
\end{matrix}
\right)
 \text{ and }
\bm{ b}^T=
\left(
\begin{matrix}
(u_{h})_{i_1}\\
(u_{h})_{i_2}\\
\cdots\\
(u_{h})_{i_{|I_i|}}
\end{matrix}
\right)
\end{equation*}
with $|I_i|$ being the  cardinality of the set $I_i$.

\begin{remark}
 As observed in \cite{DG2017},   the least-squares fitting procedure will not improve the accuracy of the solution approximation.
 We can remove one degree of freedom in the least-squares fitting procedure by assuming
 \begin{align*}
 \tilde{p}_{z_i}(z) \ =\  &u_{h,i} + \tilde{a}_2(x-x_i) + \tilde{a}_3(y-y_i)+ \tilde{a}_4(x-x_i)^2 + \\
 &\tilde{a}_5(x-x_i)(y-y_i)+ \tilde{a}_6(y-y_i)^2.
 \end{align*}
 To determine $\tilde{\bm{a}} =  (\tilde{a}_2, \tilde{a}_3, \cdots, \tilde{a}_6)^T$, we only need to solve a $5\times 5$ linear system instead of $6\times 6$ linear system.
\end{remark}

Then the recovered gradient $G_hu_h$ at the vertex $z_i$ is defined as
\begin{equation}\label{equ:grdef}
G_{h} u_h(z_i)= \nabla  p_i(z_i) = \frac{1}{h_i} \left(
\begin{matrix}
\hat{a}_2\\
\hat{a}_3
\end{matrix}
\right).
\end{equation}
Once we obtain $G_{h} u_h(z_i)$ for each $i \in I_h$,  the global recovered gradient can be interpolated as
\begin{equation}
G_{h} u_h=  \sum_{i\in I_h} G_{h} u_h(z_i) \phi_i.
\end{equation}

\begin{algorithm}
\caption{Superconvergent Gradient Recovery Procedure}
\label{alg:gr}
\begin{minipage}{\textwidth}
Let polygonal mesh $\mathcal{T}_h$   and the data (VEM solution) $(u_{h,i})_{i\in I_h}$ be given.
Then repeat steps $(1)-(3)$ for all $i\in I_h$.
\begin{itemize}
\item[(1)]
For every $z_i$, construct a local patch of elements  $\Omega_{z_i}$. Let $B_{z_i}$ be the set of vertices
in $\Omega_{z_i}$ and $I_i$ be the indexes of the $B_{z_i}$.
\item[(2)] Construct reference local patch $\hat{\Omega}_{z_i}$ and  reference set of vertices $\hat{B}_{z_i}$.
\item[(3)] Find a polynomial $\hat{p}_{z_i}$ over $\hat{\Omega}_{z_i}$ by solving the least squares problem
\begin{equation*}
\hat{p}_{z_i}=\arg\min_{\hat{p}} \sum_{j\in I_i}| {\hat{p}(\hat{z}_{i_j})-u_{h,j}}|^2\; \text{ for } \hat{p}\in \mathbb{P}_2(\hat{\Omega}_{z_i}).
\end{equation*}

\item[(4)] Calculate the partial derivatives of the approximated polynomial functions, then we have the recovered gradient at each vertex $z_i$
\begin{equation*}
G_{h} u_h(z_i)= \nabla  p_{z_i}(z_i) = \frac{1}{h_i}\nabla  \hat{p}_{z_i}(0,0).
\end{equation*}
\end{itemize}
For the recovery of the gradient $G_{h} u_h$ on the whole domain $\Omega$, we propose to interpolate the values ${G_{h} u_h(z_i)}_{i\in I_h}$ by using the standard
linear interpolation of the virtual element method.
\end{minipage}
\end{algorithm}

The recovery procedure is summarized in Algorithm \ref{alg:gr}. From Algorithm 1, it can be clearly observed that to perform the gradient recovery procedure, we actually only use the information of degrees of freedom which
is the only information directly available from the linear virtual element method.

For the purpose of theoretical analysis, we can also  treat $G_h$ as an operator from $V_h$ to $V_h\times V_h$.
It is easy to see  that $G_h$ is a linear operator.

To show the consistency of the gradient recovery operator, we begin with the following theorem:

\begin{theorem}\label{thm:ppp}
If $u$ is a quadratic polynomial on $\Omega_{z_i}$, then $G_hu(z_i) = \nabla u(z_i)$ for each $i \in I_h$.
\end{theorem}
\begin{proof}
By the definition of \eqref{equ:grdef}, we only need to  show
\begin{equation}\label{equ:identity}
\nabla p_{z_i}(z_i) = \nabla u(z_i)
\end{equation}
for all $u \in \mathbb{P}_2(\Omega_{z_i})$. To ease the presentation, we consider the least-squares fitting on the domain
$\Omega_{z_i}$ instead of the local reference domain $\hat{\Omega}_{z_i}$.  Suppose   $\{q_1(z), q_2(z), \cdots, q_6(z)\}$ is the monomial basis of
$ \mathbb{P}_2(\Omega_{z_i})$ and let $\bm{p}=( q_1(z), q_2(z), \cdots, q_6(z))$.   Then $p_{z_i}(z) = \bm{p}^T\bm{a}$ where $\bm{a}$ is determined
by the linear system
\begin{equation}\label{equ:adef}
A^TA\bm{a}=A^T\bm{b}.
\end{equation}

To prove the polynomial preserving property,  it is sufficient to show  the equation \eqref{equ:identity} is true
when $u = q_j(z)$ for $j = 1, 2, \cdots, 6$.
Let $u = q_j(z)$.  Then it implies that
\begin{equation}
\bm{ b}^T=
\left(
\begin{matrix}
(q_j)(z_{i_1})&
(q_j)(z_{i_2})&
\cdots&
(q_{j})(z_{i_{|I_i|}})
\end{matrix}
\right).
\end{equation}
It is easy to see that $A\bm{e}_j = \bm{b}$ where $\bm{e}_j$ is the $j$th canonical basis function in $\mathbb{R}^6$. Note that
$A^TA$ is nonsingular. Then  $\bm{e}_j$ is the unique  solution to the linear system \eqref{equ:lsapp}.  From \eqref{equ:polydef},  we can see
that $p_{z_i}(z) = \bm{p}(z)^T\bm{e}_j = q_j(z)$  and hence $\nabla p_{z_i}(z_i) = \nabla u(z_i)$.
Thus, for the quadratic polynomial $u$, we have $G_hu(z_i) = \nabla u(z_i)$.
\end{proof}

Theorem \ref{thm:ppp} means $G_h$ preserves quadratic polynomials at $z_i$.
Using the polynomial preserving property above, we can show the following Lemma:

\begin{lemma}
 Suppose $u_h\in V_h$,  then we have
 \begin{equation}\label{equ:bnd}
 |G_hu_h(z_i)|\lesssim h^{-1}|u_h|_{1,\Omega_{z_i}}.
 \end{equation}
\end{lemma}
\vspace{-0.2in}
\begin{proof}
 According to \eqref{equ:lsapp} and \eqref{equ:grdef},  the recovered gradient $G_hu_h(z_j)$ can be expressed as
 \begin{equation}
G_hu_h(z_i)
=\left(
\begin{matrix}
G_h^xu_h(z_i) \\
G_h^yu_h(z_i)
\end{matrix}
\right)
= \frac{1}{h_i} \left(
\begin{matrix}
\sum\limits_{j=1}^{|I_i|}\hat{c}_j^1u_{h, i_j}\\
\sum\limits_{j=1}^{|I_i|}\hat{c}_j^2u_{h, i_j}
\end{matrix}
\right),
\end{equation}
where $\hat{c}_j^k$ is independent of the mesh size.  Setting $u \equiv u_{h,i}$ in  Theorem \ref{thm:ppp} yields
  \begin{equation}
G_hu_h(z_i)
=\left(
\begin{matrix}
0 \\
0
\end{matrix}
\right).
 \end{equation}
 Combining the above two equations, we have
   \begin{equation}
 G_hu_{h,i}
= \frac{1}{h_i} \left(
\begin{matrix}
\sum\limits_{j=1}^{|I_i|}\hat{c}_j^1(u_{h, i_j}-u_{h,i})\\
\sum\limits_{j=1}^{|I_i|}\hat{c}_j^2(u_{h, i_j}-u_{h,i})
\end{matrix}
\right).
\end{equation}
For any $z_{i_j}$, we can find $z_i = z_{j_1}, z_{j_2}, \cdots, z_{j_{n_j}}=z_{i_j}$ such that the line segment $\overline{z_{j_{\ell}}z_{j_{\ell+1}}}
= e_{j_{\ell}}$
is an edge of an element $E\in \Omega_{z_i}$.  Then we can rewrite $G^x_hu_h(z_i)$ as
\begin{equation}\label{equ:grx}
G^x_hu_h(z_i)=\sum\limits_{j=1}^{|I_i|}\hat{c}_j^1\sum_{\ell=1}^{n_{j}-1}\frac{(u_{h, j_{\ell+1}}-u_{h,j_{\ell}})}{h_i}.
\end{equation}
Note that $u_h$ is virtual element function. Then we have $u_h|_{e_{j_{\ell}}}$ is a linear polynomial and hence it holds that
\begin{equation}\label{equ:directionder}
\frac{u_{h, j_{\ell+1}}-u_{h,j_{\ell}}}{|e_{j_\ell}|} = \frac{\partial u_{h}}{\partial t_{j_{\ell}}}  \le |\nabla  u_h|_{0, \infty, e_{j_{\ell}}},
\end{equation}
where $t_{j_{\ell}}$ is the unit vector in the direction from $z_{j_{\ell}}$ to $z_{j_{\ell+1}}$.  Substituting \eqref{equ:directionder}
into \eqref{equ:grx} gives
\begin{equation}
|G^x_hu_h(z_i)|= \sum_{E\in \Omega_{z_i}}\sum_{e\in \mathcal{E}_E}\frac{|e|}{h_i}|\nabla u_h|_{0, \infty, e}.
\end{equation}
Since $u_h|_e$ is a linear polynomial, the inverse inequality \cite{BS2008, Ci2002} is applicable,  which implies
\begin{equation}
|\nabla u_h|_{0, \infty, e} \lesssim |e|^{-\frac{3}{2}}|| u_h||_{0,  e} .
\end{equation}
By the scaled trace inequality \cite{BGS2017}, we have
\begin{equation}
|| u_h||_{0,  e} \lesssim h^{-\frac{1}{2}}||u_h||_{0,  E}  +  h^{\frac{1}{2}}  | u_h|_{1,  E} .
\end{equation}
Combining the above estimates and noticing that $\frac{|e|}{h_i}$ is bounded by a fixed constant  using the assumption (ii) on the mesh $\mathcal{T}_h$, we have
\begin{align*}
|G^x_hu_h(z_i)|\lesssim&  \sum_{E\in \Omega_{z_i}}\left(h^{-2}||u_h||_{0,  E}  +  h^{-1} |u_h|_{1,  E}\right)\\
\lesssim &h^{-2}||u_h||_{0, \Omega_{z_i}}+h^{-1}|u_h|_{1,\Omega_{z_i}}.
\end{align*}
Let $\bar{u}_{h} = \frac{1}{|\Omega_{z_i}|}\int_{\Omega_{z_i}}u_hdz$.   Setting $u_h\equiv \bar{u}_{h}$ in Theorem \ref{thm:ppp} implies $G_h\bar{u}_{h}(z_i)=(0,0)^T$.
Replacing $u_h$ by $u_h- \bar{u}_{h} $ in the above estimate, we have
\begin{align*}
|G^x_hu_h(z_i)= & |G^x_h(u_h- \bar{u}_{h} )(z_i)|\\
\lesssim &h^{-2}||u_h- \bar{u}_{h} ||_{0, \Omega_{z_i}}+h^{-1}|u_h- \bar{u}_{h} |_{1,\Omega_{z_i}}\\
\lesssim & h^{-1}|u_h|_{1,\Omega_{z_i}},
\end{align*}
where we have used the scaled Poincar\'e-Freidrichs  inequality in \cite{BGS2017}.

Similarly, we can establish the same error bound for $G_h^yu_h(z_i)$. Thus, the estimate  \eqref{equ:bnd} is true.
\end{proof}

Based on the above lemma, we can establish the local boundedness in $L_2$ norm:
\begin{theorem}
 Suppose $u_h\in V_h$,  then for any $E\in \mathcal{T}_h$, we have
 \begin{equation}\label{equ:lcoalbnd}
 ||G_hu_h||_{0, E} \lesssim |u_h|_{1,\Omega_{E}},
 \end{equation}
where $\Omega_E = \bigcup\limits_{z\in E\cap\mathcal{N}_h}\Omega_{z}$.
\end{theorem}
\begin{proof}
Let $I_E$ be the index set of $E\cap\mathcal{N}_h$. Since $\{\phi_i\}_{i\in I_h}$ is the canonical basis for $V_h$,  we have
 \begin{equation*}
\begin{split}
  ||G_hu_h||_{0, E} \lesssim & |E|^{\frac{1}{2}}\sum_{j=1}^{|I_E|} |G_hu_h(z_{i_j})|\\
  \lesssim  &\sum_{j=1}^{|I_E|} |E|^{\frac{1}{2}}h^{-1}|u_h|_{1, \Omega_{z_{i_j}}}\\
   \lesssim & |u_h|_{1,\Omega_{E}},
\end{split}
\end{equation*}
where we have used the fact that $|E|^{\frac{1}{2}}h^{-1}$ is bounded by a fixed constant.
\end{proof}

As a direct consequence, we can prove the following corollary.
\begin{corollary}\label{cor:bdp}
 Suppose $u_h\in V_h$,  then we have
 \begin{equation}\label{equ:bdp}
 ||G_hu_h||_{0, \Omega} \lesssim |u_h|_{1,\Omega}.
 \end{equation}
\end{corollary}
Corollary \ref{cor:bdp} implies that $G_h$ is a linear bounded operator from $V_h$ to $V_h\times V_h$.
Now, we are in the perfect position to  present the consistency result of $G_h$.

\begin{theorem}\label{thm:consist}
Suppose $u\in H^3(\Omega_{E})$, then we have
$$
\|G_hu-\nabla u\|_{0,E}\lesssim h^2\|u\|_{3,\Omega_{E}}.
$$
\end{theorem}
\vspace{-0.2in}
\begin{proof}
Define $$\mathcal{F}(u)=\|G_hu-\nabla u\|_{0,E}.$$
By the boundedness of $G_h$, it is easy to see that
\begin{equation*}
\begin{split}
 \mathcal{F}(u) \le& \|G_hu\|_{0, E} + \|\nabla u\|_{0,E}\\
 \lesssim &|u|_{1, \Omega_E}.
\end{split}
\end{equation*}
The polynomial property of the gradient recovery operator $G_h$ implies $G_hp=\nabla p$ for any $p\in\mathbb{P}_2(\Omega_E)$. Thus we have
$$\mathcal{F}(u+p)=\mathcal{F}(u).$$
By the Brambler-Hilbert Lemma\cite{BS2008, Ci2002}, we obtain
  \begin{displaymath}
\mathcal{F}(u)\lesssim h^2\|u\|_{3,\Omega_E}.   \end{displaymath}
\end{proof}

Theorem \ref{thm:consist} implies the gradient recovery operator is consistent in the sense that the recovered gradient using the exact solution is superconvergent to
the exact gradient at   a rate of $\mathcal{O}(h^2)$.

\section{Adaptive Virtual Element Method}
\label{sec:avem}
The adaptive virtual element method is summarized as a loop of the following steps:
\begin{equation}
\bf
Solve \rightarrow Estimate\rightarrow Mark \rightarrow Refine
\end{equation}

Starting from an initial polygonal  mesh, we solve the equation by using the linear virtual element method. Once the virtual element solution is available, we need to  design a fully computational
{\it a posteriori} error estimator using only the virtual element solution.  This step is vital for the adaptive virtual element
method because it determines the performance of the adaptive algorithm.  In this paper, we introduce  a recovery-based {\it a posteriori}
error estimator using the proposed gradient recovery method, which we elaborate  in the coming subsection.

\subsection{Recovery-based {\it a posteriori} error estimator}
\label{ssec:estimator}

Provided that the recovered gradient is reconstructed, we are ready to present the recovery-type {\it a posteriori } error estimator for virtual element methods.
However, for virtual element methods,  their basis functions are not explicitly constructed which means that both $G_hu_h$ and $\nabla u_h$ are not  computable quantities.
To overcome this difficulty, we propose to use $\Pi^0_{E}G_hu_h$ and $\nabla\Pi^0_{E} u_h$, which are computable.
We define a local {\it a posteriori} error estimator on each  polygonal element E as
\begin{equation}\label{equ:localind}
\eta_{h,E} =
\|\Pi^0_{E}G_hu_h - \nabla  \Pi^0_{E}u_h\|_{0, E},
\end{equation}
and the corresponding global error estimator  as
\begin{equation}\label{equ:globalind}
\eta_h = \left( \sum\limits_{E\in \mathcal{T}_h}\eta_{h,E}^2\right)^{1/2}.
\end{equation}

To measure the performance of  the {\it a posteriori} error estimator  \eqref{equ:localind} or \eqref{equ:globalind}, we introduce
the effective index
\begin{equation}\label{equ:effidx}
\kappa_h = \frac{\|\Pi^0_{E}G_hu_h -\nabla\Pi^0_{E}  u_h\|_{0, \Omega}}
{\|\nabla u  -  \nabla \Pi^0u_h\|_{0, \Omega}},
\end{equation}
which is computable when the exact solution $u$ is provided.

For {\it a posteriori} error estimators,  the ideal case we expect is the so-called asymptotic exactness.
\begin{definition}\label{def:exact}
The {\it a posteriori} error  estimator \eqref{equ:localind} or \eqref{equ:globalind} is said to be  asymptotically exact if
\begin{equation}
\lim_{h\rightarrow 0} \kappa_h = 1.
\end{equation}
\end{definition}
A series of benchmark numerical examples in the next section indicate the recovery-based {\it a posteriori} error estimator \eqref{equ:localind} or \eqref{equ:globalind}
is asymptotically exact for the linear virtual element method, which distinguishes it from the residual-type {\it a posteriori} error
estimators for virtual element methods in the literature \cite{CGPS2017, BM2015,BB2017}.

\subsection{Marking strategy}
Once the recovery-type {\it a posteriori} error estimator  $\eqref{equ:localind}$ is available, we pick up a set of elements to be refined. This process is called marking.
There are several different marking strategies.  In this paper, we adopt the bulk marking strategy proposed by D{\"o}rfler \cite{Do1996}.
 Given  a constant $\theta\in (0, 1]$, the bulk strategy is to find $\mathcal{M}_h\subset \mathcal{T}_h$ such  that
 \begin{equation}
\left( \sum_{E\in \mathcal{M}_h} \eta_{h,E}^2\right)^{\frac{1}{2}} \le \theta \left( \sum_{E\in \mathcal{T}_h} \eta_{h,E}^2\right)^{\frac{1}{2}}.
\end{equation}
In general, the choice of $\mathcal{M}_h$ is not unique.  We select $\mathcal{M}_h$ such that the cardinality of $\mathcal{M}_h$ is smallest.

\subsection{Adaptive mesh refinement}
  One of the main advantages of virtual element methods is their flexibility in local mesh refinement. Virtual element methods
allow that elements  can have an arbitrary number of edges and two  edges  can be collinear. These advantages enable virtual element
methods  to naturally handle  hanging nodes.  A polygon with a hanging node is just a polygon  that has an extra edge collinear with another edge.
It avoids artificial refinement of the unmarked neighborhood in the classical adaptive finite element methods.
Take  the polygon in Figure \ref{fig:polygon} as example.  It is a  pentagon with five vertices $V_1, V_2, \cdots, V_5$.
But there are three hanging nodes $V_6, V_7, V_8$ which are generated by the refinement of its neighborhood element.
In the virtual element method, we can treat the pentagon with three hanging nodes as an octagon with eight vertices
$V_1, V_2, \cdots, V_8$.   Note that  in the octagon, there are four edges are collinear  which is allowed in the
virtual element method.

In the paper, we adopt the same way to refine a polygon as in \cite{CGPS2017, AFMD2018}.
We divide a polygon into several sub-polygons  by connecting its barycenter
to each planar edge center.  Note that two or more edges collinear to each other are treated as one planar edge.
We take the polygon in Figure \ref{fig:polygon} as an example again.   The refinement of the polygonal is
illustrated in Figure \ref{fig:polygon} by the dashed lines.   Note that the four edges collinear to each other
are viewed as one edge $\overline{V_5V_1}$.  Thus, in the refinement, we bisect $\overline{V_5V_1}$ instead of the four collinear edges.

\begin{figure}
   \centering
  {\includegraphics[width=0.6\textwidth]{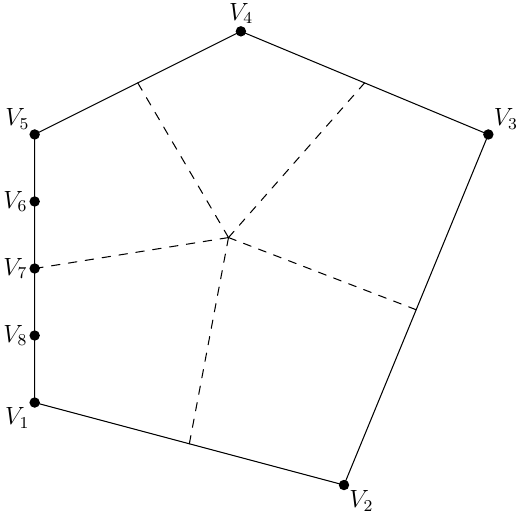}}
    \caption{Illustration of handling hanging nodes in the virtual element method  and local refinement of   polygonals with collinear edges}\label{fig:polygon}
\end{figure}

\section{Numerical Results}
\label{sec:ne}
In this section,  we present several numerical examples to demonstrate our numerical discoveries.
The first example is to illustrate the superconvergence of the proposed gradient recovery.
The other examples are to numerically  validate the asymptotic exactness of the
recovery-based {\it a posteriori } error estimator.

In the virtual element method,  the basis functions are never explicitly constructed and  the numerical solution is unknown inside elements.
In the computational test, we shall use the  projection $\Pi_h^0u_h$ to compute different errors instead of using $u_h$.
In addition,  all the convergence rates are illustrated in term of the degrees of freedom (DOF). In two dimensional cases,
DOF $\approx h^2$ and the corresponding convergence rates in term of the mesh size $h$ are doubled of what we plot in the graphs.

\subsection{Test case 1: smooth problem}   In this example, we consider the following homogeneous elliptic equation
\begin{equation}
-\Delta u = 2\pi^2\sin(\pi x)\sin(\pi y), \quad \text{ in  } \Omega = (0,1) \times (0,1).
\end{equation}
The exact solution is $u(x,y) = \sin(\pi x)\sin(\pi y)$.

\begin{figure}
   \centering
   \subcaptionbox{$\mathcal{T}_{h,1}$\label{fig:quad_mesh}}
  {\includegraphics[width=0.32\textwidth]{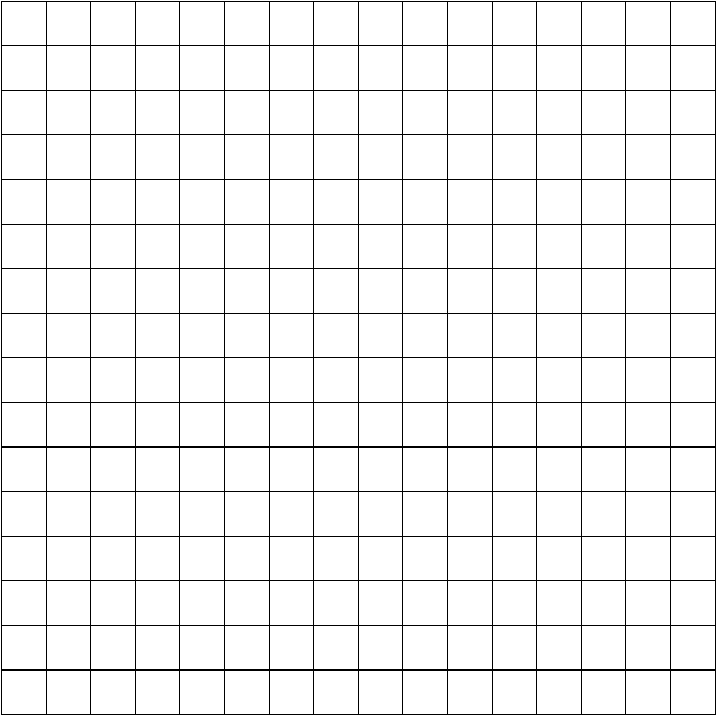}}
  \subcaptionbox{$\mathcal{T}_{h,2}$\label{fig:hexagonal_mesh}}
   {\includegraphics[width=0.32\textwidth]{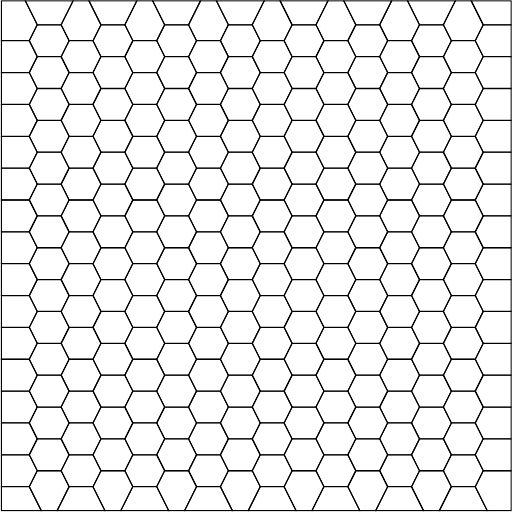}}
  \subcaptionbox{$\mathcal{T}_{h,3}$\label{fig:nonconvex_mesh}}
  {\includegraphics[width=0.32\textwidth]{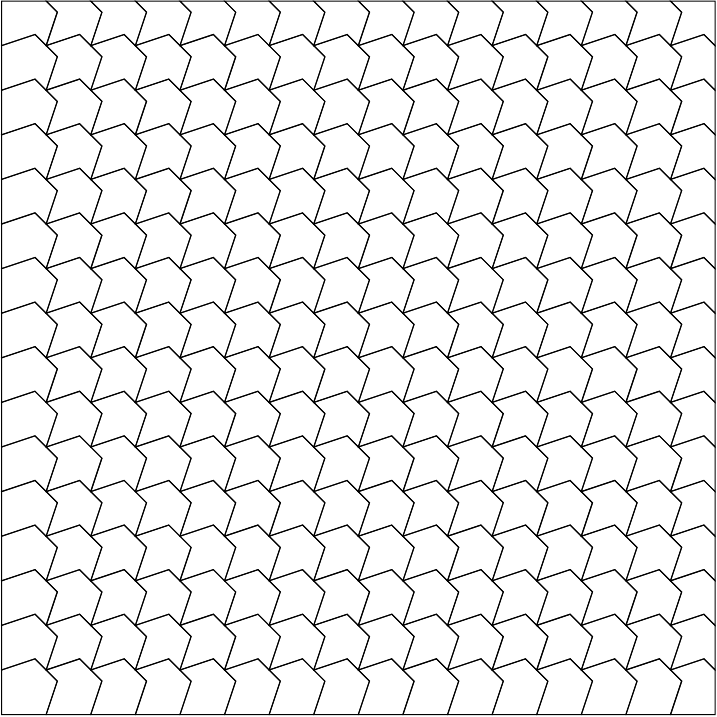}}\\
   \subcaptionbox{$\mathcal{T}_{h,4}$\label{fig:randquad_mesh}}
  {\includegraphics[width=0.32\textwidth]{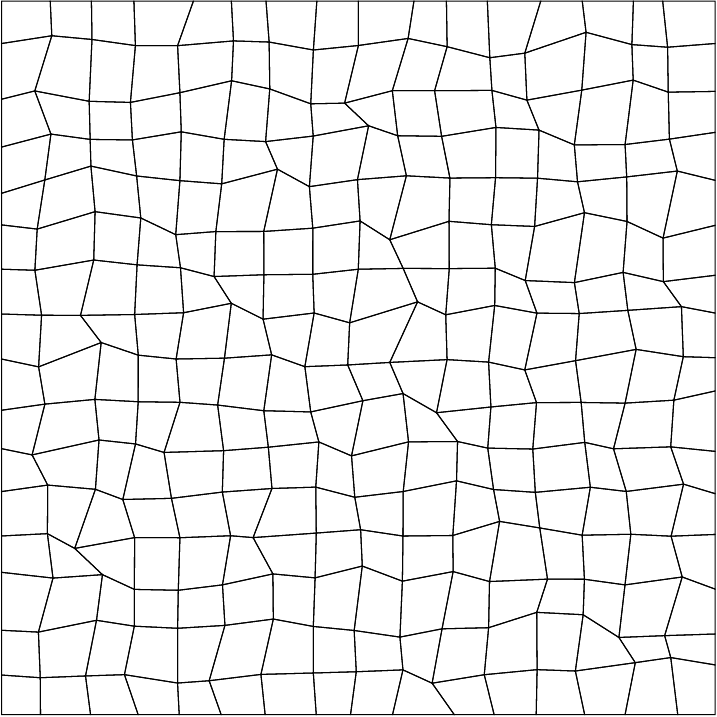}}
  \subcaptionbox{$\mathcal{T}_{h,5}$\label{fig:brezzi_mesh}}
   {\includegraphics[width=0.32\textwidth]{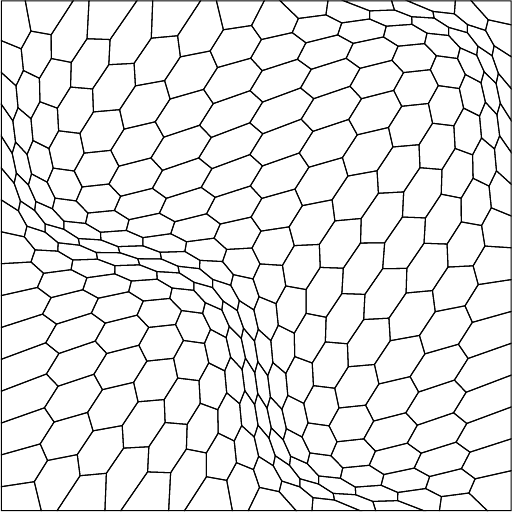}}
  \subcaptionbox{$\mathcal{T}_{h,6}$\label{fig:voronoi_mesh}}
  {\includegraphics[width=0.32\textwidth]{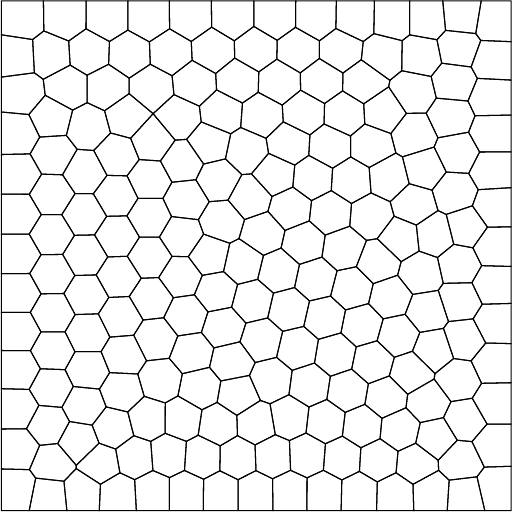}}
   \caption{Sample meshes for numerical tests: (a)  uniform quadrilateral mesh; (b) structured hexagonal mesh;
   (c)  concave mesh; (d)  general quadrilateral mesh; (e)  hexagonal mesh ;
   (f)  Voronoi mesh.}\label{fig:mesh}
\end{figure}


In this test,   we adopt six different types of  meshes  to numerically show the superconvergence of the proposed gradient recovery method.
The first level of each type of meshes  are plotted in  Figure \ref{fig:mesh}.   The first type  of mesh $\mathcal{T}_{h,1}$  is just the uniform square  mesh.
The second type  of mesh $\mathcal{T}_{h,2}$ is uniform hexagonal mesh.
The third type  of mesh $\mathcal{T}_{h,3}$ is uniform non-convex  mesh.
$\mathcal{T}_{h,4}$ is generated by adding random perturbation to the  mesh $\mathcal{T}_{h,1}$.
The fifth type of mesh $\mathcal{T}_{h,5}$ is generated by applying  the following coordinate transform
\begin{equation*}
\begin{split}
 x= \hat{x}+\frac{1}{10}\sin(2\pi \hat{x})\sin(2\pi \hat{y}),\\
 y = \hat{y}+\frac{1}{10}\sin(2\pi \hat{x})\sin(2\pi \hat{y});\\
\end{split}
\end{equation*}
to the uniform hexagonal  mesh $\mathcal{T}_{h,2}$.  The sixth type of mesh $\mathcal{T}_{h,6}$ is smoothed Voronoi mesh generated by Polymesher\cite{TPPM2012}.

In addition to the discrete $H_1$ semi-error $\|\nabla u - \nabla \Pi^0_h u_h\|_{0, \Omega}$  and the recovered error $\|\nabla u -\Pi^0_h G_hu_h\|_{0, \Omega}$,
we also consider the error $\|\nabla u_h - \nabla u_I\|_{0, \Omega} $.   In this paper, we approximate  $\|\nabla u_h - \nabla u_I\|_{0, \Omega} $ by a computable  quantity $ \sqrt{(\bm{u_h}-\bm{u_I})^TA_h(\bm{u_h}-\bm{u_I})}$
where $A_h$ is the stiffness matrix,   $u_I$ is the interpolation of $u$ into the virtual element space $V_h$, and $\bm{u_h} $ (or $\bm{u_I}$)
is a vector of value of  $u_h$ (or $u_I$) on the degrees of freedom.  The error  $\|\nabla u_h - \nabla u_I\|_{0, \Omega}$ plays an important role in the study of superconvergence
for gradient recovery methods in the classical  finite element methods \cite{BX2003, XZ2004}.  We say the gradient of the numerical solution is
superclose to the gradient of the interpolation of the exact solution if $\|\nabla u_h - \nabla u_I\|_{0, \Omega}\lesssim \mathcal{O}(h^{1+\rho})$
for some $0<\rho\le 1$.  The supercloseness result is a sufficient but not necessary condition to prove the superconvergence of gradient recovery methods \cite{BX2003, XZ2004, ZN2005, GZ2015}.

We plot the rates of convergence for the above three different errors in Figure \ref{fig:err}.  As predicted in \cite{BBCMMR2013}, the discrete $H_1$ semi-error
decays at  a rate of $\mathcal{O}(h)$ for all the above six different types of meshes. Concerning the the error  $\|\nabla u_h - \nabla u_I\|_{0, \Omega}$,
we can only  observe $\mathcal{O}(h^2)$ supercloseness on two structured convex meshes and the transformed meshes $\mathcal{T}_{h,5}$.
It is not surprising since the supercloseness depends strongly on the symmetry of meshes even on triangular meshes, see \cite{BX2003, XZ2004}.
But the recovered gradient is superconvergent to the exact gradient at  a rate of $\mathcal{O}(h^2)$ on all the above meshes including  meshes with non-convex elements.   The above numerical observation is summarized in Table \ref{tab:error}.

\begin{figure}
   \centering
   \subcaptionbox{$\mathcal{T}_1$\label{fig:quad_err}}
  {\includegraphics[width=0.47\textwidth]{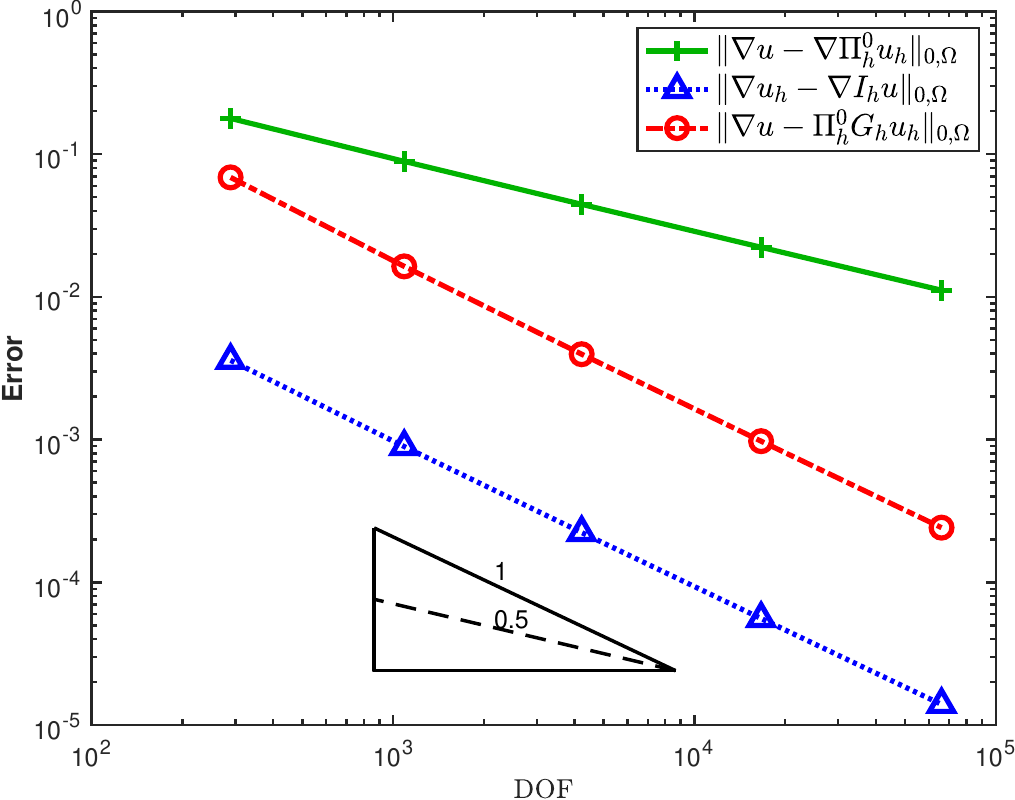}}
  \subcaptionbox{$\mathcal{T}_2$\label{fig:hexagonal_err}}
   {\includegraphics[width=0.47\textwidth]{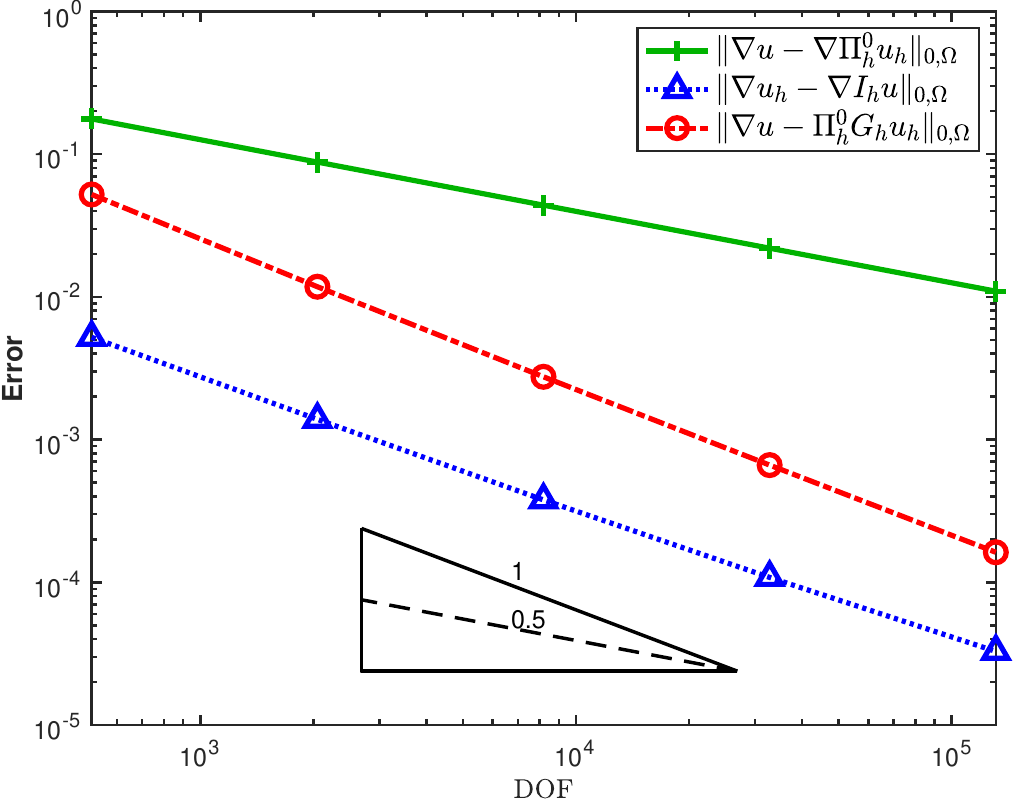}}\\
  \subcaptionbox{$\mathcal{T}_3$\label{fig:nonconvex_err}}
  {\includegraphics[width=0.47\textwidth]{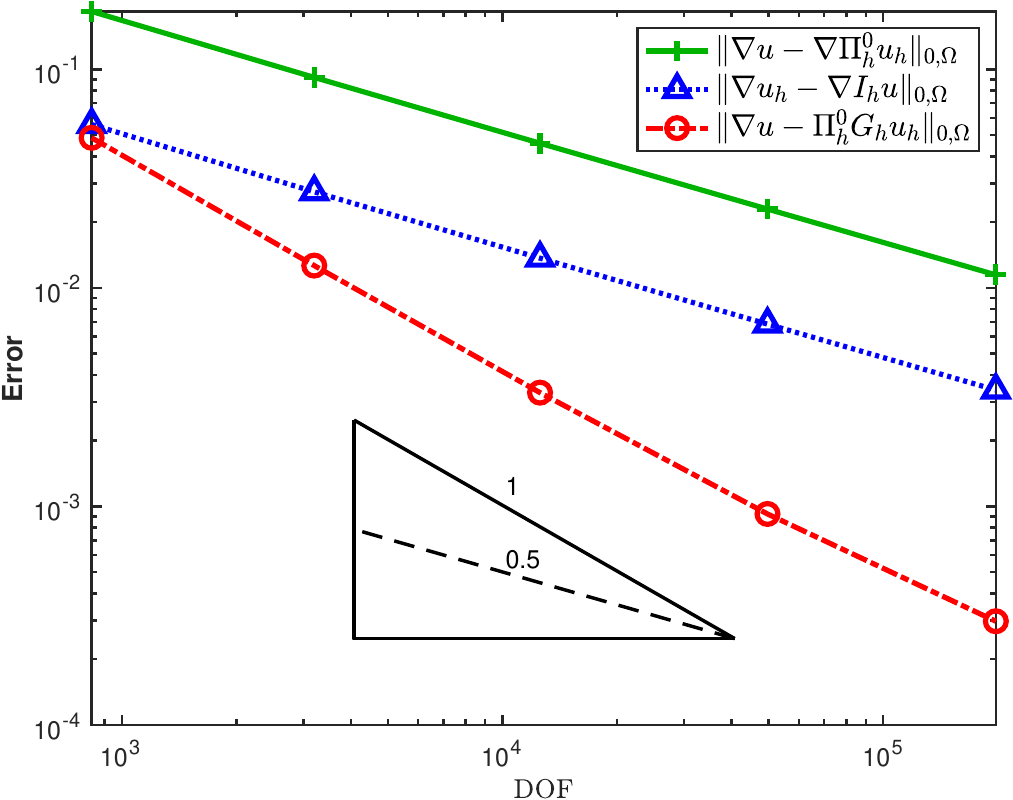}}
   \subcaptionbox{$\mathcal{T}_4$\label{fig:randquad_err}}
  {\includegraphics[width=0.47\textwidth]{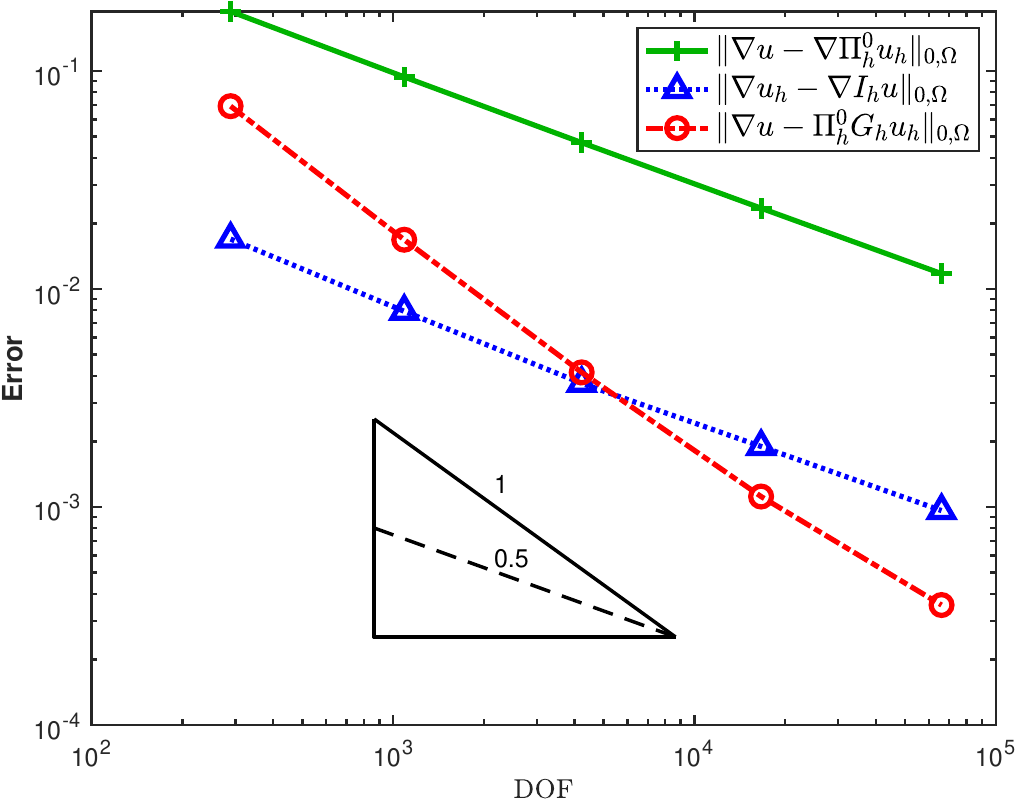}}\\
  \subcaptionbox{$\mathcal{T}_5$\label{fig:brezzi_err}}
   {\includegraphics[width=0.47\textwidth]{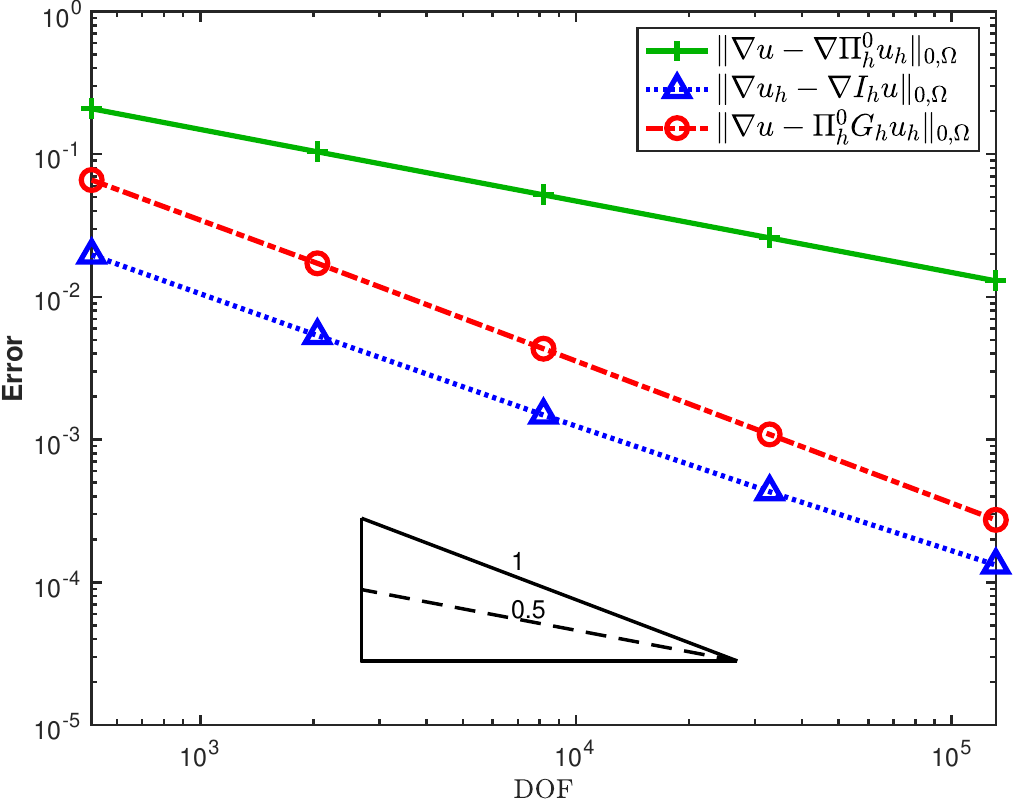}}
  \subcaptionbox{$\mathcal{T}_6$\label{fig:voronoi_err}}
  {\includegraphics[width=0.47\textwidth]{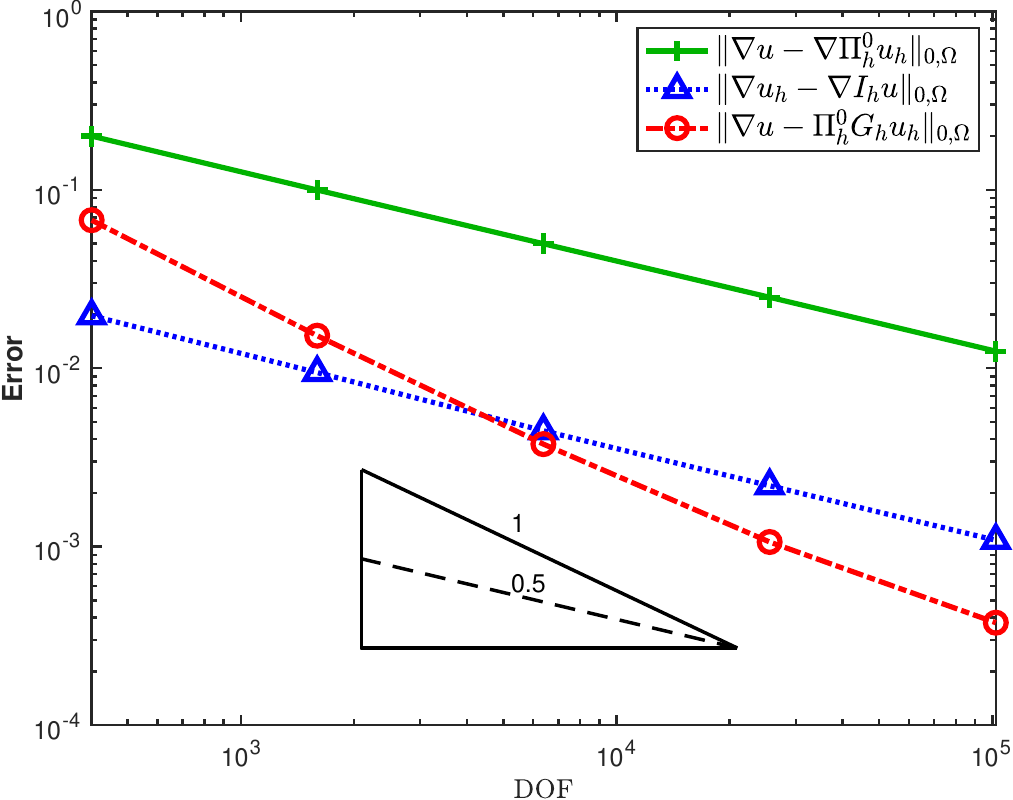}}
   \caption{Sample errors for numerical tests: (a) on structured quadrilateral mesh; (b) on  structured hexagonal mesh;
   (c) on structure concave mesh; (d) on unstructured quadrilateral mesh; (e) on unstructured hexagonal mesh;
   (f) on unstructured Voronoi mesh.}\label{fig:err} 
\end{figure}

\begin{table}[h!]
  \begin{center}
    \caption{Summary of numerical results on the  six different types of  meshes}
    \label{tab:error}
    \begin{tabular}{|c|c|c|c|c|c|} \hline
      Mesh  Type& $\|\nabla u  - \Pi^0 \nabla u_h\|_{0, \Omega}$ & $\|\nabla u_h - \nabla u_I\|_{0, \Omega}$
      & $\|\nabla u - \Pi^0_{E} G_hu_h\|_{0, \Omega}$\\
      \hline
      $\mathcal{T}_{h,1}$& $\mathcal{O}(h)$ & $\mathcal{O}(h^2)$& $\mathcal{O}(h^2)$\\ \hline
            $\mathcal{T}_{h,2}$& $\mathcal{O}(h)$& $\mathcal{O}(h^2)$& $\mathcal{O}(h^2)$\\ \hline
      $\mathcal{T}_{h,3}$& $\mathcal{O}(h)$& $\mathcal{O}(h)$& $\mathcal{O}(h^2)$ \\ \hline
      $\mathcal{T}_{h,4}$& $\mathcal{O}(h)$& $\mathcal{O}(h)$&$\mathcal{O}(h^2)$  \\ \hline
      $\mathcal{T}_{h,5}$& $\mathcal{O}(h)$&$\mathcal{O}(h^2)$& $\mathcal{O}(h^2)$ \\ \hline
      $\mathcal{T}_{h,6}$& $\mathcal{O}(h)$& $\mathcal{O}(h)$ & $\mathcal{O}(h^2)$ \\ \hline
    \end{tabular}
  \end{center}
\end{table}

\subsection{Test case 2: L-shaped domain problem}  In this example, we consider the Laplace equation
\begin{equation*}
-\Delta u = 0,
\end{equation*}
on the L-shaped domain  $\Omega= [-1, 1] \times [-1, 1] \backslash (0, 1) \times (-1, 0)$. The exact solution
is $u = r^{2/3}\sin(2\theta/3)$ in polar coordinate.
The corresponding  boundary condition  is  computed from the exact solution $u$.
Note the exact solution $u$  has a singularity at the origin.

\begin{figure}
   \centering
   \subcaptionbox{\label{fig:lshape_init}}
  {\includegraphics[width=0.47\textwidth]{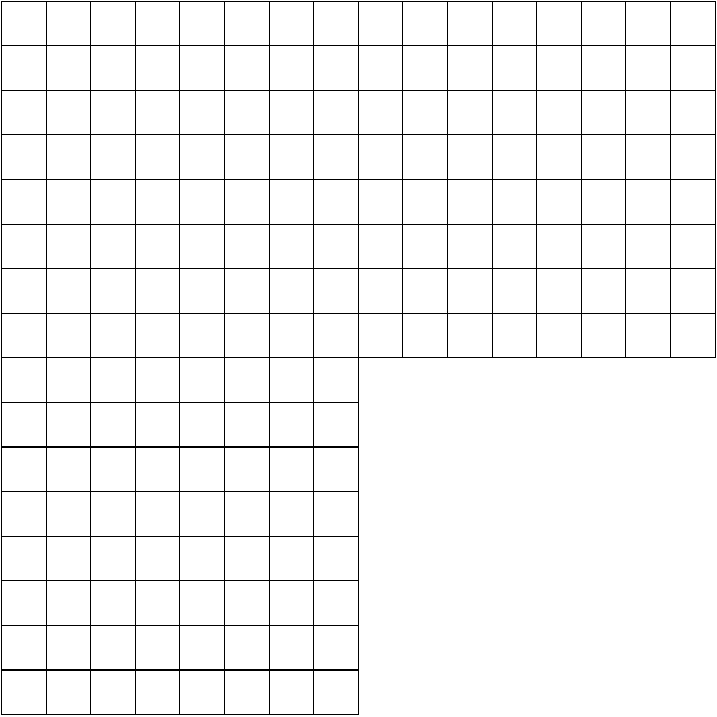}}
  \subcaptionbox{\label{fig:lshape_adaptive}}
   {\includegraphics[width=0.47\textwidth]{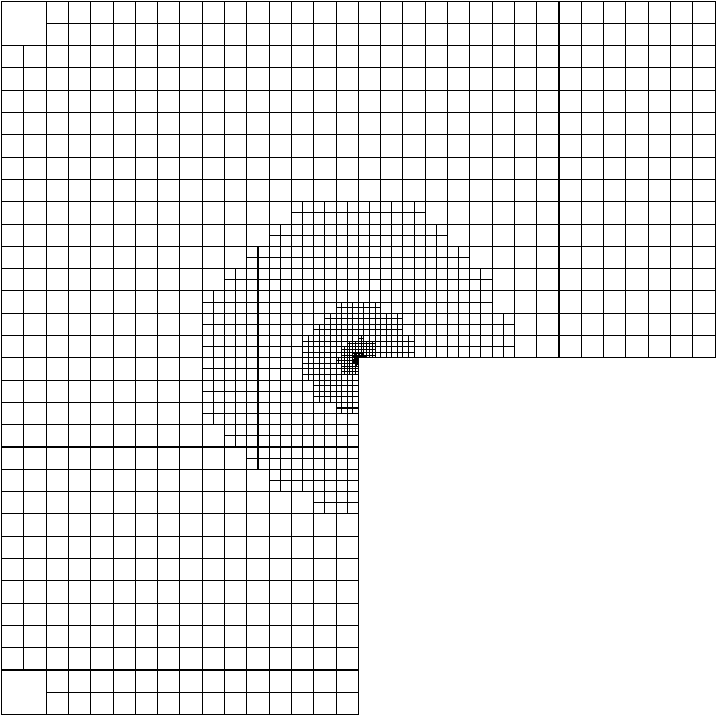}}
   \caption{Meshes for test case 2: (a) Initial mesh; (b) Adaptively refined mesh.}
   \label{fig:lshape_mesh}
\end{figure}

\begin{figure}
   \centering
   \subcaptionbox{\label{fig:lshape_error}}
  {\includegraphics[width=0.48\textwidth]{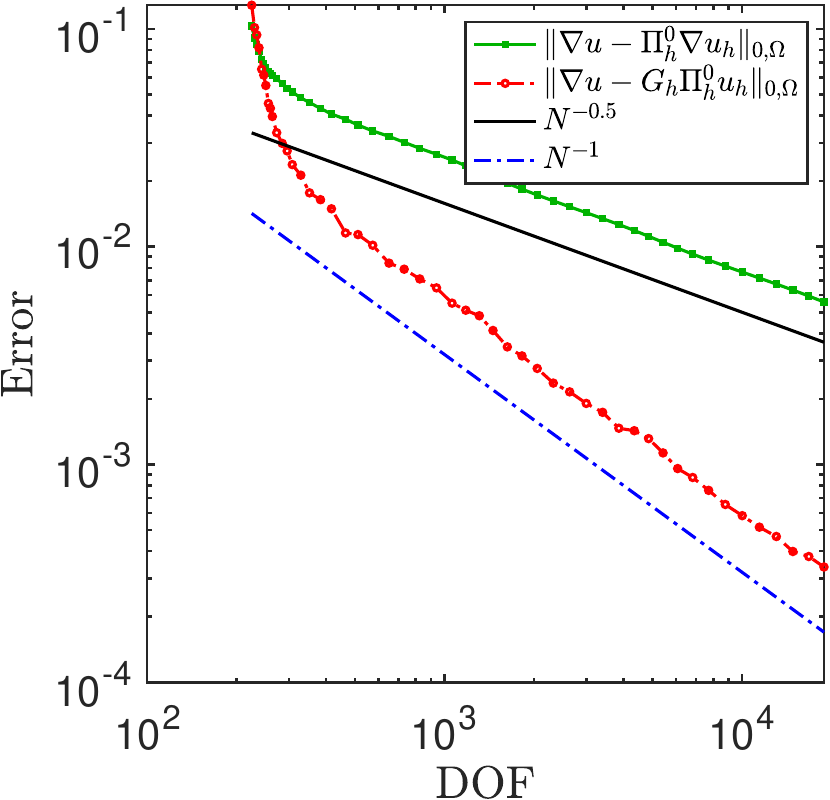}}
  \subcaptionbox{\label{fig:lshape_ind}}
   {\includegraphics[width=0.5\textwidth]{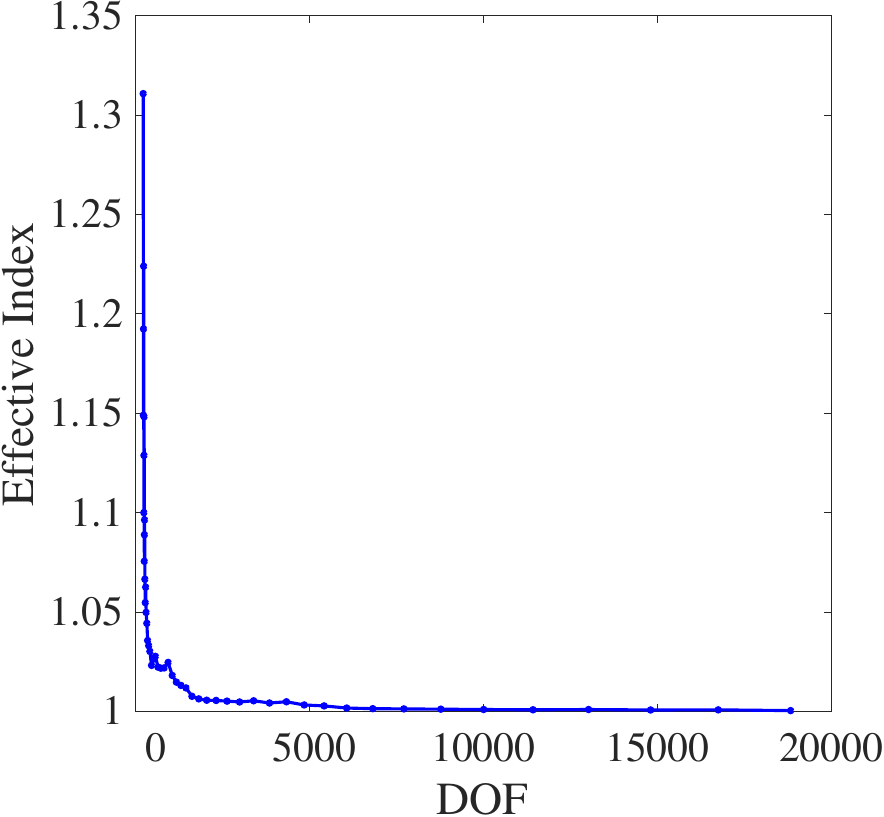}}
   \caption{Numerical result for test case 2: (a) Numerical errors; (b) Effective index.}
   \label{fig:lshape_result}
\end{figure}

To resolve the singularity, we
use the adaptive virtual element method described in Section \ref{sec:avem}.
The initial mesh is plotted in Figure \ref{fig:lshape_init}, which is a uniform mesh consisting of
square elements.   In Figure \ref{fig:lshape_adaptive}, we show the corresponding adaptively
refined mesh.  It is not hard to see that the refinement is conducted near the singular point.

In Figure \ref{fig:lshape_error}, we  depict the rates of convergence for discrete $H_1$ semi-error  $\|\nabla u - \nabla \Pi^0_h u_h\|_{0, \Omega}$
and the discrete recovery error $\|\nabla u -\Pi^0_h G_hu_h\|_{0, \Omega}$.  From the plot, we can clearly observe $\mathcal{O}(h)$ optimal convergence
for the virtual element gradient and $\mathcal{O}(h^2)$ superconvergence for the recovered gradient for the adaptive virtual element method.
It means the recovery-based {\it a posteriori} error estimator \eqref{equ:localind} is robust. To quantify the performance of the error estimator, we draw
the effective index \eqref{equ:effidx} in Figure \ref{fig:lshape_ind}.   It shows that the effective index converges to 1 rapidly after a few iterations.   It means
the {\it a posteriori } error estimator is asymptotically exact  as defined in  Definition \ref{def:exact}.

\begin{figure}
   \centering
   \subcaptionbox{\label{fig:Gaussian_init}}
  {\includegraphics[width=0.47\textwidth]{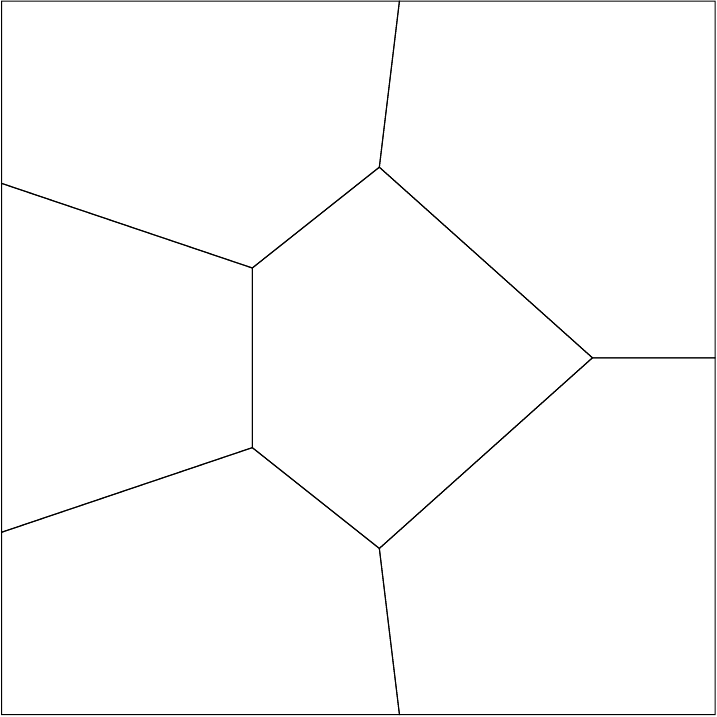}}
  \subcaptionbox{\label{fig:Gaussian_adaptive}}
   {\includegraphics[width=0.47\textwidth]{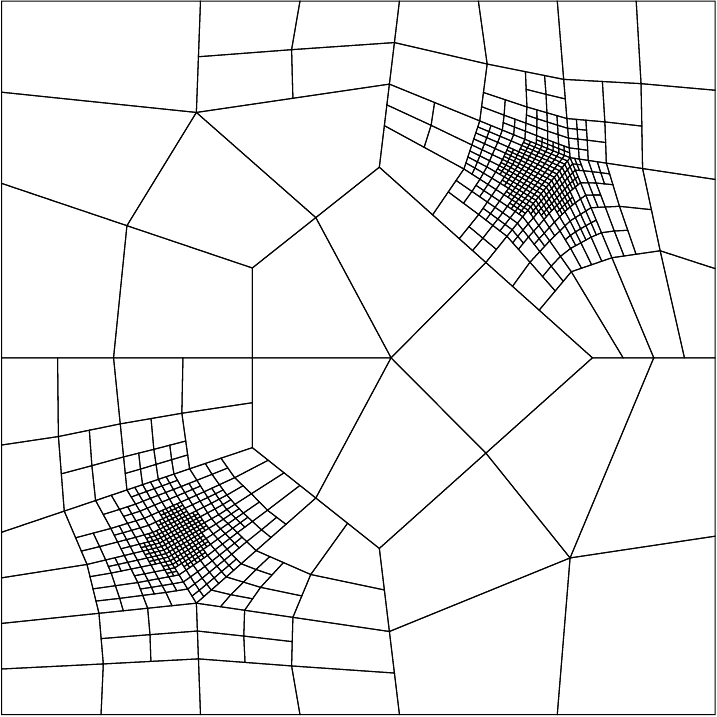}}
   \caption{Meshes for test case 3: (a) Initial mesh; (b) Adaptively refined mesh.}
   \label{fig:Gaussian_mesh}
\end{figure}

\begin{figure}
   \centering
   \subcaptionbox{\label{fig:gaussian_error}}
  {\includegraphics[width=0.49\textwidth]{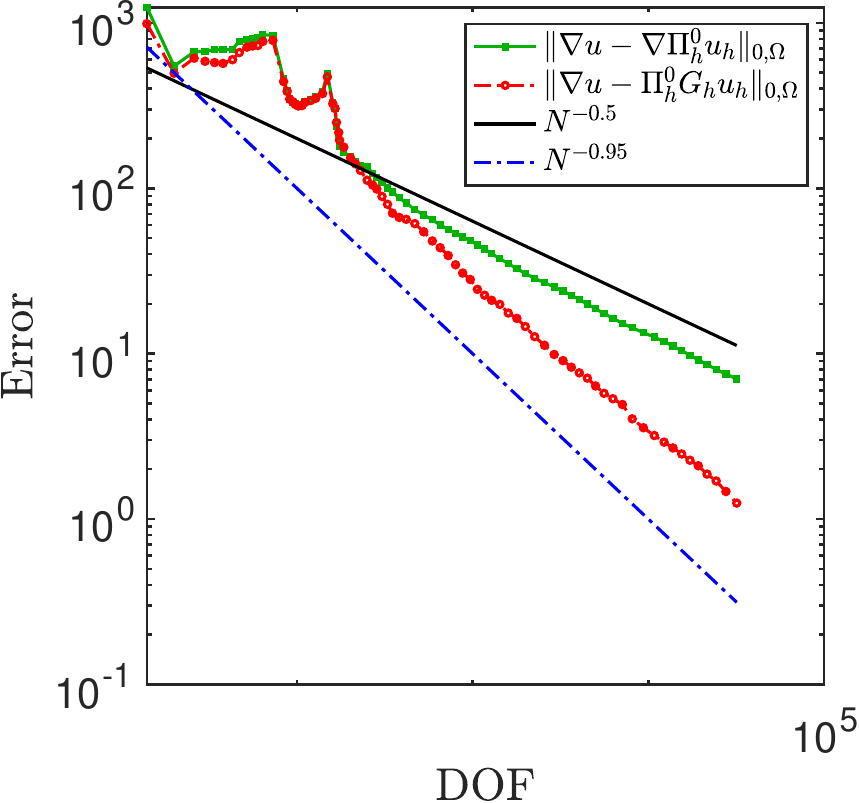}}
  \subcaptionbox{\label{fig:gaussian_ind}}
   {\includegraphics[width=0.485\textwidth]{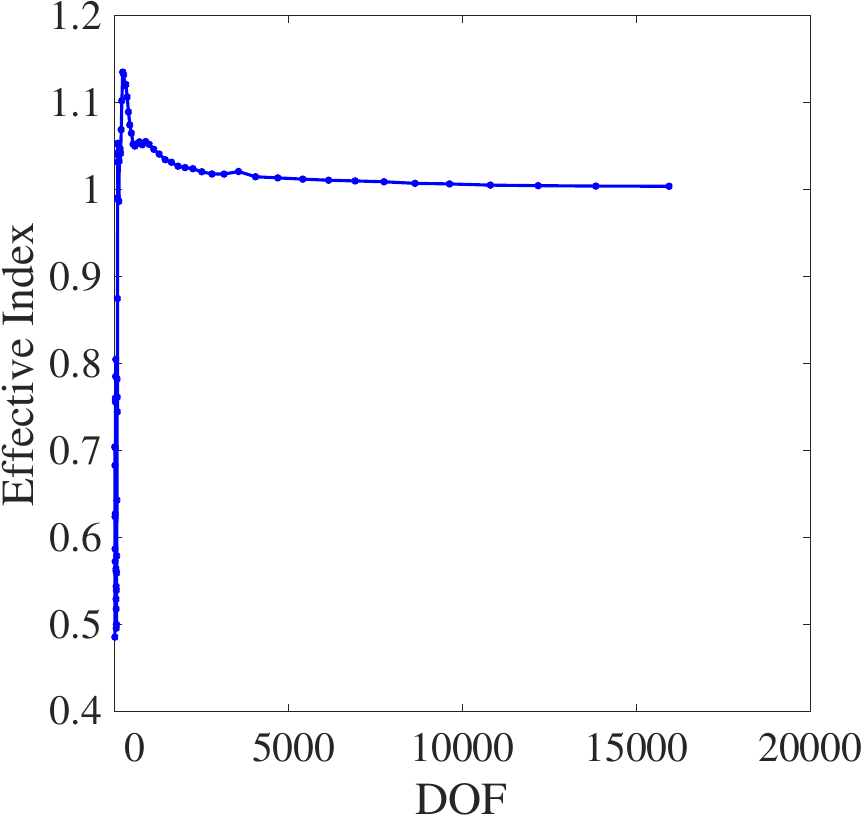}}
   \caption{Numerical result for test case 3: (a) Numerical errors; (b) Effective index.}
   \label{fig:gaussian_result}
\end{figure}

\subsection{Test case 3:  problem with  two Gaussian surfaces}  Consider  the Poisson equation \eqref{equ:model}  on the unit square with
the exact solution
 $$
 u(x,y) = \frac{1}{2\pi \sigma }\left[ e^{-\frac{1}{2}\left( \frac{x-\mu_1}{\sigma}\right)^2}
 e^{-\frac{1}{2}\left( \frac{y-\mu_1}{\sigma}\right)^2} +
 e^{-\frac{1}{2}\left( \frac{x-\mu_2}{\sigma}\right)^2}
 e^{-\frac{1}{2}\left( \frac{y-\mu_2}{\sigma}\right)^2}
 \right]
 $$
 as in \cite{AFMD2018}.   In this test,  the standard deviation is $\sigma = \sqrt{10^{-3}}$ and
 the two means are $\mu_1= 0.25$ and $\mu_2=0.75$.

The difficulty of this problem is the existence of two Gaussian surfaces, where the solution has  a  fast decay.  Here we adopt  the same
initial mesh as in  \cite{AFMD2018}, see Figure \ref{fig:Gaussian_init}.  It is a polygonal mesh, which does not resolve the Gaussian surfaces.
In Figure \ref{fig:Gaussian_adaptive}, we show the adaptively refined mesh.   Clearly, the mesh is refined near the location of
the Gaussian surfaces.  In Figure \ref{fig:gaussian_error}, we present the numerical errors.  Similar to test case 2, we can
observe the desired optimal and superconvergent  results.    Moreover,  the asymptotic exactness of the error estimator
\eqref{equ:localind} is numerically verified  in  Figure \ref{fig:gaussian_ind} by the fact  that the effective index is convergent to one.

\begin{figure}
   \centering
   \subcaptionbox{\label{fig:sharp_init}}
  {\includegraphics[width=0.47\textwidth]{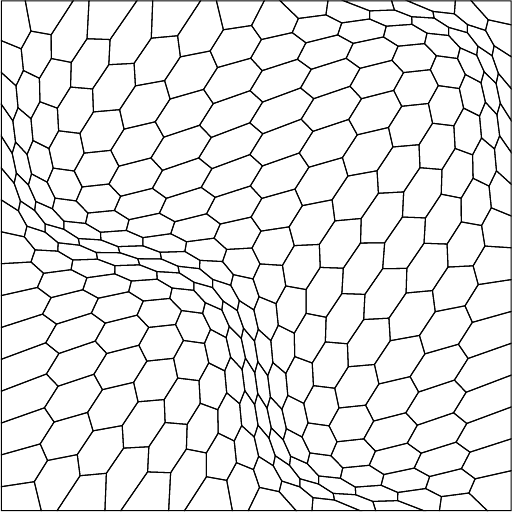}}
  \subcaptionbox{\label{fig:sharp_adaptive}}
   {\includegraphics[width=0.47\textwidth]{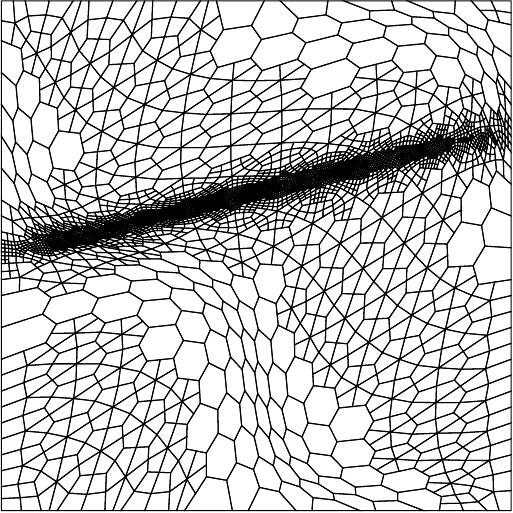}}
   \caption{Meshes for test case 4: (a) Initial mesh; (b) Adaptively refined mesh.}
   \label{fig:sharp_mesh}
\end{figure}

\begin{figure}
   \centering
   \subcaptionbox{\label{fig:sharp_error}}
  {\includegraphics[width=0.48\textwidth]{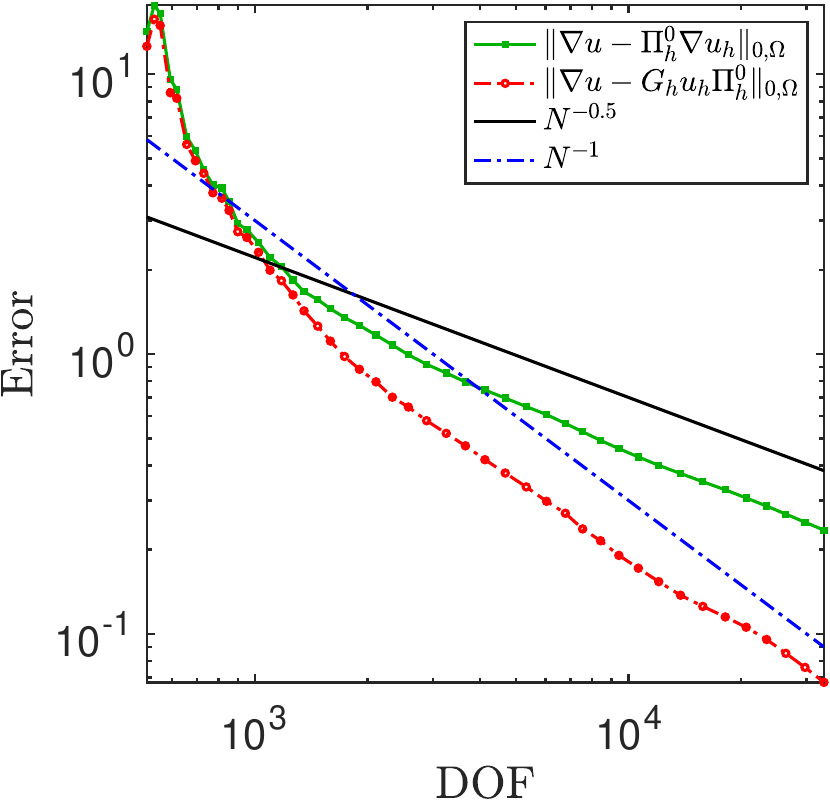}}
  \subcaptionbox{\label{fig:sharp_ind}}
   {\includegraphics[width=0.5\textwidth]{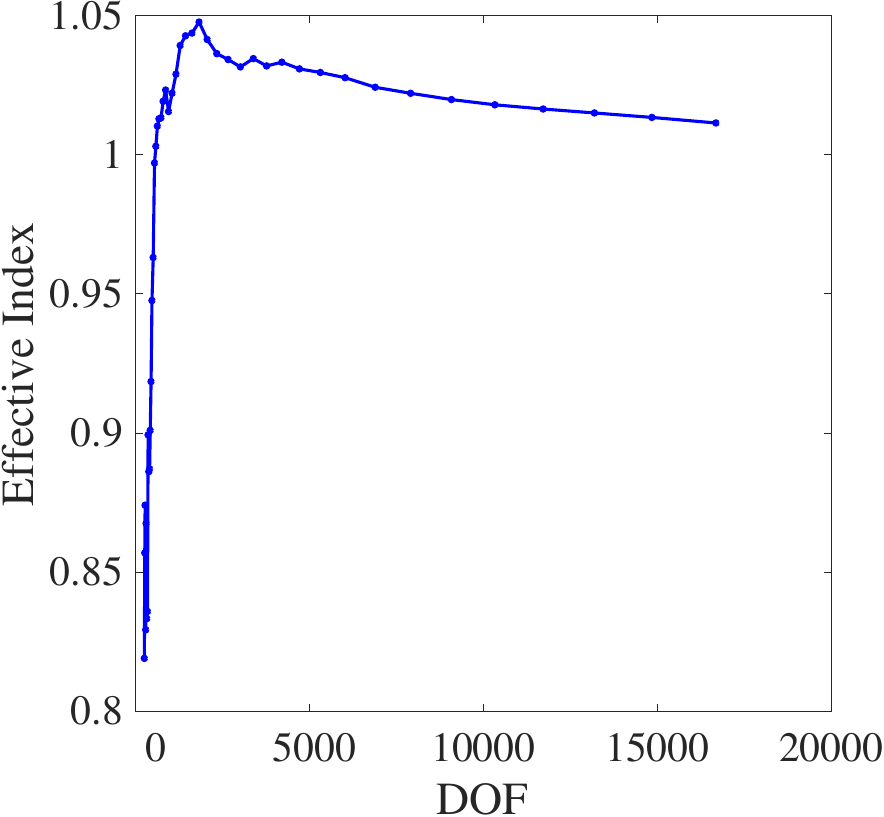}}
   \caption{Numerical result for test case 4: (a) Numerical errors; (b) Effective index.}
   \label{fig:sharp_result}
\end{figure}

\subsection{Test case 4:  problem with sharp interior layer}   As in \cite{CGPS2017, BM2015}, we
consider  the Poisson equation \eqref{equ:model} and \eqref{equ:bndcond} on the unit square   with a sharp interior layer.
The exact solution is
 $$
 u(x,y) = 16x(1-x)y(1-y)\arctan(25x-100y+25).
 $$

The initial mesh is the transformed hexagonal mesh  $\mathcal{T}_{h,5}$ as in Test Case 1, which is shown
in Figure \ref{fig:sharp_init}.   It is an unstructured polygonal mesh. The interior sharp layer is totally unresolved by the initial mesh which causes the major difficulty.
  Figure \ref{fig:sharp_adaptive} is the mesh generated by the adaptive virtual element method prescribed in Section \ref{sec:avem}.
  It is obvious that the mesh is refined to resolve the interior layer as expected.

  In Figure \ref{fig:sharp_result}, we present the qualitative results.   As anticipated, the desired $\mathcal{O}(h)$ optimal convergence rate for the virtual element gradient and
  $\mathcal{O}(h^2)$ superconvergence rate for the recovered gradient can be numerically observed.  Also, the limit of the effective index is numerically proved to be one,
  which validates the asymptotic exactness of the error estimator \eqref{equ:localind}.

\section{Conclusion}
\label{sec:con}

In this paper, a superconvergent gradient recovery method for the virtual element methods is introduced.  The proposed post-processing technique
 uses only the degrees of freedom which are the only data directly obtained from the virtual element methods.  It generalizes the idea of polynomial preserving recovery \cite{ZN2005, NZ2005} to general polygonal meshes.   Theoretically, we prove the proposed gradient recovery method is bounded and consistent.   It meets the standard of a good gradient recovery technique in \cite{AO2000}.   Numerically,  we validated the superconvergence of the recovered gradient using the virtual element solution on several different types of general polygonal meshes including  meshes with non-convex elements.   In the future, it would be
interesting to present a theoretical proof of those superconvergence for the virtual element method.

  Its capability of serving as  {\it a posteriori} error estimators is also exploited.   The asymptotic exactness of the recovery-based {\it a posteriori} error estimator is numerically verified by three benchmark problems.  To the best of our knowledge, it is the first recovery-based {\it a posteriori} error estimator for the virtual element methods.   Compared to the existing  residual type {\it a posteriori}  error estimators,  it has several advantages:  (i) it is  simple in both the idea and implementation,
  which makes it more realistic for practical applications;  (ii) the unique  characterization of the error estimator  is asymptotic exactness, which  prevails over all other {\it a posteriori} error estimators in the literature  for the virtual element  methods.

The application of gradient recovery is not limited to adaptive methods. It has also been applied to many other fields, like enhancing eigenvalues\cite{NZZ2006, NZ2012, GZZ2017} and designing new numerical methods for higher order PDEs\cite{CGZZ2017, GZZ2018, GZZ2018b, XGZ2019}.  We will make use of those advantages of gradient recovery to study more interesting real application problems in  future work.

\section*{Ackowledgement}
The authors thank the anonymous referees for their comments and suggestions which significantly improve the quality of this paper.

\bibliographystyle{siam}
\bibliography{mybibfile}

\end{document}